\documentclass{amsart}

\usepackage{amssymb,mathtools}
\usepackage{graphicx} %
\usepackage{commath}
\usepackage[capitalize]{cleveref}
\crefformat{equation}{(#2#1#3)}
\usepackage{subcaption}

\usepackage{amsrefs}

\usepackage{enumerate}
\usepackage{stackrel}
\def\ddefloop#1{\ifx\ddefloop#1\else\ddef{#1}\expandafter\ddefloop\fi}

\def\ddef#1{\expandafter\def\csname v#1\endcsname{\ensuremath{\boldsymbol{#1}}}}
\ddefloop ABCDEFGHIJKLMNOPQRSTUVWXYZabcdefghijklmnopqrstuvwxyz\ddefloop  %

\usepackage{amsmath}
\DeclareMathOperator*{\spn}{span}

\newtheorem{theorem}{Theorem}[section]
\newtheorem{lemma}[theorem]{Lemma}

\theoremstyle{definition}

\newcommand{\LL}{\mathcal{L}}
\newcommand{\M}{\mathcal{M}}

\newcommand{\me}{\mathrm{e}} %

\theoremstyle{remark}
\newtheorem{remark}[theorem]{Remark}

\numberwithin{equation}{section}

\usepackage{xcolor}

\begin{document}

\title[Computing Coercivity Constants]{On Computing Coercivity Constants in Linear Variational Problems through Eigenvalue Analysis}

\author{Peter Sentz}
\address{Department of Computer Science, University of Illinois at Urbana-Champaign}
\email{sentz2@illinois.edu}

\author{Jehanzeb Hameed Chaudhry}
\address{Department of Mathematics, University of New Mexico}
\email{jehanzeb@unm.edu}
\urladdr{https://math.unm.edu/\string~jehanzeb}

\author{Luke N. Olson}
\address{Department of Computer Science, University of Illinois at Urbana-Champaign}
\email{lukeo@illinois.edu}
\urladdr{http://lukeo.cs.illinois.edu}


\begin{abstract}
  In this work, we investigate the convergence of numerical approximations to coercivity constants of variational problems.  These constants are essential components of rigorous error bounds for reduced-order modeling; extension of these bounds to the error with respect to exact solutions requires an understanding of convergence rates for discrete coercivity constants.
  The results are obtained by characterizing the coercivity constant as a spectral value of a self-adjoint linear operator; for several differential equations, we show that the coercivity constant is related to the eigenvalue of a compact operator.  For these applications, convergence rates are derived and verified with numerical examples.
\end{abstract}

\maketitle

\section{Introduction}\label{sec:intro}

A key ingredient in error estimates for variational problems is knowledge of the so-called \textit{coercivity} or \textit{inf-sup} constants of the continuous problem.  We consider variational problems of the form: find $u \in V$ such that
\begin{equation}\label{eq:variational_equation}
a(u,v) = F(v),\quad \forall v \in V,
\end{equation}
where $V$ is a Hilbert space of functions, $a:V\times V\to\mathbb{R}$ is a continuous bilinear form, and $F\in V'$ is a continuous linear functional.  A sufficient condition for~\cref{eq:variational_equation} to be well-posed is the existence of a constant $\alpha > 0$ such that
\begin{equation}\label{eq:coercive_form}
	a(u,u)\geq \alpha \|u\|_V^2,\quad \forall u\in V
\end{equation}
i.e., the bilinear form is coercive.  The largest positive number $\alpha$ for which~\cref{eq:coercive_form} holds is the \textit{coercivity constant}.  The norm in~\Cref{eq:coercive_form} is induced by the inner product on $V$~---~throughout this paper, for an arbitrary Hilbert space $W$, we denote the (induced) norm as $\|u\|_W$ and the inner product as $(u,u)_W$.  While not a necessary condition for well-posedness, coercive forms are an important class of variational problems.  An estimate of the  error often relies on a lower bound of the coercivity constant $\alpha$~\cites{huynh2007successive,rozza2008reduced,nguyen2010reduced,buffa2012priori};
however, this value is rarely computable.  Numerical approximations of $\alpha$ converge rapidly in cases where the constant can be characterized as the eigenvalue of a compact operator~\cite{boffi2010finite}.  Establishing similar convergence rates is essential for understanding error bounds that make use of the coercivity constant, and is the focus of this work.

Error bounds of this type are relevant in reduced-order modeling (ROM), which replaces an expensive ``full-order'' model with one that is computationally tractable, yet still resolving certain features in the original problem~\cites{huynh2011high,benner2015survey,rozza2005shape,quarteroni2015reduced,rozza2008reduced,quarteroni2007numerical,quarteroni2007numerical,manzoni2012shape}.
For instance, partial differential equations (PDEs) depending on a set of parameters require quick and/or repeated evaluation in a variety of situations, including PDE-constrained optimization, real-time analysis of physical systems, and parameter estimation.  In order for the reduced-order solutions to be useful, an indicator or estimate of the error is essential.  Estimates of the error with respect to the full-order numerical solution have been developed~\cite{rozza2008reduced}, allowing for a controlled loss of accuracy and providing information on the fidelity of a given reduced model.  This is useful in instances where high accuracy of the full-order model is guaranteed, such as PDEs featuring smooth solutions or a sufficiently fine computational mesh.

Yet, the error in a full-order solution often depends on the value of the problem's parameters, and the assumption of highly accurate full-order approximations can result in overly optimistic error estimates for the reduced solution with respect to the underlying analytical solution.  With this in mind, recent work~\cite{chaudhry2020leastsquares} has made progress in developing error estimates for reduced-order models measured with respect to  exact solutions.  In~\cite{yano2018reduced}, an error estimate for exact solutions is developed where the coercivity constant is bounded below by the solution of a linear programming problem; however, the standard function space norm is replaced by a parameter-dependent ``energy norm''.  Establishing lower bounds of eigenvalues for variationally posed problems is critically important in other applications as well; see~\cites{sebestova2014two,hu2014lower2,hu2014lower, betcke2011numerical, cances2017guaranteed, you2019guaranteed}.

In this work, we consider problems where the coercive bilinear form has the form
\begin{equation}\label{eq:bilinear}
	a(u,v) = (\LL u, \M v)_Y,
\end{equation}
where $\LL$ and $\M$ are linear operators with domain $V$, densely embedded in a Hilbert space $X$, and each operator maps $V$ into another Hilbert space $Y$:
\begin{equation}
	\LL,\M\,:\, V\rightarrow Y.
\end{equation}
We show that the coercivity constant associated with~\cref{eq:bilinear} is characterized as a spectral value of a self-adjoint linear operator involving $\LL$ and $\M$.  For the case when this operator has the form $I + C$, where $C$ is a compact operator, and $I$ is the identity map,
 we present results on the convergence rate of numerical approximations to the coercivity constant.  We demonstrate that operators of this type arise in variationally-posed differential equations using several examples; the convergence theory is verified numerically.

The paper is organized as follows.  In~\Cref{sec:func_analysis}, we summarize relevant results from functional analysis and spectral theory.  In~\Cref{sec:coercivity}, we show the coercivity constant is the minimum of the Rayleigh quotient of a self-adjoint linear operator $K$, thus demonstrating that the coercivity constant is a spectral value.  In~\Cref{sec:numeric_convergence} we establish results on the rate of convergence when $K - I$ is a compact operator.  We apply our results to several differential equations~\Cref{sec:apps}, and demonstrate that the connection with a compact operator arises for these problems; we verify the convergence results for these examples numerically in~\Cref{sec:numerics}.  In~\Cref{sec:conclusion}, we present conclusions and possible directions for future work.

\section{Selected Results in Functional Analysis}\label{sec:func_analysis}

Before proceeding to the main contributions of the paper, we establish notation
 and review several
relevant results from functional analysis  that are used throughout this work.

\subsection{Riesz map, Gelfand triple, and adjoint operators}
For a Hilbert space $W$ and its dual, $W'$, we denote the canonical isometric isomorphism as $R_W: W\to W'$,
\begin{equation}\label{eq:riesz_map}
\langle R_W u, v \rangle_{W'\times W} = (u,v)_W,\quad \forall u,v\in W.
\end{equation}
In~\cref{eq:riesz_map} we use the notation of duality pairing; i.e., $\langle f,v\rangle_{W'\times W}\coloneqq f(v)$ for any $f\in W',\  v\in W$.  The operator $R_W$ is known as the \textbf{Riesz map}.
We note that $R_W$ is isometric in the sense that
\begin{equation}\label{riesz_norm}
\|R_W u\|_{W'} \coloneqq \sup_{0\neq v\in W} \frac{\left|  \langle R_W u, v \rangle_{W'\times W}  \right|}{\|v\|_W} = \|u\|_W,
\end{equation}
a property we  make use of later,
and  $W'$ is a Hilbert space with inner product
\begin{equation}\label{eq:dual_inner_product}
(f,g)_{W'} \coloneqq (R_W^{-1}f, R_W^{-1}g)_W.
\end{equation}

In the following sections, we make use of the concept of a \textbf{Gelfand triple}, which we now recall.  Let $V, X$ be Hilbert spaces such that $V \subset X$, with continuous and dense embedding.  More precisely, the inclusion map $v\mapsto v$ from $V$ into $X$ is continuous (bounded):
\begin{equation}
\exists C > 0 \text{  s.t.  }  \|v\|_X \leq C\|v\|_V, \quad \forall v\in V,
\end{equation}
and $V$ is a dense subspace of $X$ under the $X$-norm:
\begin{equation}
\overline{V}^{\|\cdot\|_X} = X.
\end{equation}
Here, $\overline{V}^{\|\cdot\|_X}$ denotes the closure of $V$ under the $X$-norm.

Identifying $X$ with its dual through the isometric isomorphism~\cref{eq:riesz_map}, we obtain
\begin{equation}\label{inclusion_chain}
V \subset X \subset V',
\end{equation}
with each embedding continuous and dense.  In this case, $X$ is known as a \textbf{pivot space}, and is regarded as a dense subspace of $V'$.  The term ``pivot space'' will also be used to describe any Hilbert space $Y$ that is identified with its dual $Y'$ through~\cref{eq:riesz_map}, without necessarily being part of a Gelfand triple.

One example of a Gelfand triple is a familiar result from the theory of Sobolev spaces:
\begin{equation}
H_0^1(\Omega) \subset L^2(\Omega) \subset H^{-1}(\Omega).
\end{equation}
Going forward, $V$ and $X$ will be understood as members of a Gelfand triple~\cref{inclusion_chain}.  Other common names for~\cref{inclusion_chain} include \textbf{rigged Hilbert space} or \textbf{equipped Hilbert space}.  For more details, see~\cites{aubin2007approximation,berezanskiui1968expansions,braess2007finite,gel2014generalized}.

\begin{remark}
It is important to note that once we consider $X$ to be a pivot space, the possibility to consider a (proper) subspace $V$ as a pivot space is no longer available.  That is, we cannot write $V = V'$, because by~\cref{inclusion_chain}, $V$ is a proper dense subspace of $V'$.
\end{remark}

As a consequence of~\cref{inclusion_chain}, if $f \in X$ and $v \in V$, then the expressions $\langle f, v\rangle_{V' \times V}$ and $(f,v)_X$ are synonymous.  That is, we simply write
\begin{equation}\label{eq:inner_product_duality}
\langle f, v\rangle_{V' \times V} =  (f,v)_X.
\end{equation}
 In particular, if $u \in V \subset X$, then for all $v \in V$,
\begin{equation}
\langle u, v\rangle_{V' \times V} =  (u,v)_X \neq (u,v)_V = \langle R_V u, v\rangle_{V' \times V},
\end{equation}
where $R_V: V \to V'$ is the Riesz map (see~\cref{eq:riesz_map}).  This is a restatement of the fact that $V$ cannot be considered a pivot space simultaneously with $X$.

If $A: W_1 \to W_2$ is a bounded linear map between normed spaces $W_1$ and $W_2$, its \textbf{adjoint} $A'$ is a linear map from $W_2' \to W_1'$ satisfying
\begin{equation}\label{eq:banach_adjoint}
	\left\langle A' f, w\right\rangle_{W_1'\times W_1} = \left\langle f, Aw\right\rangle_{W_2'\times W_2},\quad \forall f\in W_2',\  w\in W_1.
\end{equation}
$A'$ is a bounded linear operator that satisfies $\|A'\| = \|A\|$~\cite{kreyszig1978introductory}.

In the context of the Gelfand triple~\cref{inclusion_chain}, consider a bounded linear operator $A: V\to X$.  Since $X$ is identified with its dual space,~\cref{eq:inner_product_duality} holds, and thus the adjoint $A'$ is a bounded linear map satisfying
\begin{subequations}\label{eq:banach_adjoint_pivot}
	\begin{align}
		&A': X \to V'\\
		&\left\langle A' f, v\right\rangle_{V'\times V} = (f, Av)_X,\quad \forall f\in X,\ v\in V.
	\end{align}
\end{subequations}
We note that~\cref{eq:banach_adjoint_pivot} is valid for any Hilbert space $V$ and pivot space $X$.

In addition to the adjoint defined in~\cref{eq:banach_adjoint}, bounded linear operators between Hilbert spaces posses another adjoint, the \textbf{Hilbert-adjoint}.  Specifically, if $W_1$ and $W_2$ are Hilbert spaces, and $A: W_1 \to W_2$ is a bounded linear operator, the Hilbert-adjoint $A^*$ is a linear map from $W_2 \to W_1$ satisfying
\begin{equation}\label{eq:hilbert_adjoint}
	(A^* w_2, w_1)_{W_1} = (w_2, Aw_1)_{W_2},\quad \forall w_1\in W_1,\  w_2\in W_2.
\end{equation}
The Hilbert-adjoint is bounded, with $\|A^*\| = \|A\|$.

If $V$ is a Hilbert space, a bounded linear operator $A: V\to V$ is \textbf{self-adjoint} if $A = A^*$.

\begin{remark}
	The adjoints defined in~\cref{eq:banach_adjoint} and~\cref{eq:hilbert_adjoint} are well-known in functional analysis, but there is no generally accepted notation~\cites{kreyszig1978introductory,conway2019course,rudin1991functional,quarteroni2015reduced,aubin2007approximation,oden2012introduction}; we have adopted the notation from~\cite{conway2019course}.  If $V,X$ are arbitrary Hilbert spaces and $A: V\to X$ is a bounded linear operator, the two adjoints are related by
	\begin{equation}
		R_V A^* = A'R_X,
	\end{equation}
	where $R_V,R_X$ are the Riesz maps for $V$ and $X$ respectively, defined in~\cref{eq:riesz_map}.  If $X$ is a pivot space, as in~\cref{eq:banach_adjoint_pivot}, then the relationship simplifies to
	\begin{equation}
		R_V A^* = A'.
	\end{equation}
	See~\cites{oden2012introduction,aubin2007approximation} for details.
\end{remark}

\begin{remark}
There is a third notion of adjoint for densely-define linear operators on Hilbert spaces~\cites{kreyszig1978introductory,conway2019course,arnold2018finite}.  We do not use this definition to derive our results.  See Appendix A for more comments on this adjoint.
\end{remark}

\subsection{Spectral Theory of Bounded Linear Operators}\label{sec:spectral_theory}

In~\cref{sec:numeric_convergence}, we will apply theorems from~\cites{descloux1978spectrala, descloux1978spectralb} to establish convergence of numerical approximations of the coercivity constant in~\cref{eq:coercive_form}.  We state these theorems after recalling some notation and results from the spectral theory of bounded linear operators; see~\cites{chatelin2011spectral, riesz2012functional} for details.  For simplicity, we restrict attention to bounded linear operators on a Hilbert space, rather than the more general setting of Banach spaces.

Let $W$ be a complex Hilbert space and $A: W\to W$ a bounded linear operator.  For any $z\in \mathbb{C}$, define an associated operator $A_z = A - zI$, where $I$ is the identity map on $W$.  $z$ is a regular value of $A$ if
\begin{enumerate}
	\item $A_z:W \to W$ is a bijection.
	\item $R(A,z):=(A_z)^{-1}: W \to W$ is bounded.
\end{enumerate}
The set of all regular values of $A$ is called the \textbf{resolvent set} of $A$, and is denoted by $\rho(A)$.  The complement $\sigma(A):=\mathbb{C}\setminus\rho(A)$ is called the \textbf{spectrum} of $A$; a number $\lambda \in \sigma(A)$ is a \textbf{spectral value}.

In this work, we restrict attention to a subset of the spectrum, called the \textbf{point spectrum} $\sigma_p(A)$.  A complex number $\lambda$ belongs to $\sigma_p(A)$ if and only if $A_\lambda$ is not injective.  In this case, there exists a non-zero $w\in W$ such that
\begin{equation}
	(A - \lambda I)w = A_\lambda w = 0 \iff Aw = \lambda w.
\end{equation}
Thus, elements of the point spectrum are \textbf{eigenvalues} of $A$; this portion of the spectrum is relevant to the numerical approximation of the coercivity constant.  If $W$ is infinite-dimensional, there may exist spectral values which are not eigenvalues; see~\cites{chatelin2011spectral,kreyszig1978introductory,atkinson2005theoretical,rudin1991functional} for characterization of other spectral values.

Of particular importance is the case where $\lambda \in \sigma_p(A)$ is an isolated eigenvalue of finite algebraic multiplicity.  $\lambda$ is \textbf{isolated} if there exists a bounded, open set $U \subset\mathbb{C}$ such that $U \cap \sigma(A) = \{\lambda\}$, with boundary $\partial U$ that is rectifiable and contained in $\rho(A)$.   To define algebraic multiplicity, we first define the \textbf{spectral projection} corresponding to $\lambda$.  If $U$ is an open set satisfying the hypotheses above, the spectral projection is a linear operator $E_\lambda : W\to W$ defined by
\begin{equation}\label{eq:spectral_projection}
	E_{\lambda} = -\frac{1}{2\pi i} \int_{\partial U} R(A,z)\dif z = -\frac{1}{2\pi i} \int_{\partial U} (A - zI)^{-1}\dif z.
\end{equation}

The definition~\cref{eq:spectral_projection} is independent of the choice of the open set $U$.  The operator $E_\lambda$ is a projection onto its closed range $M \coloneqq E_\lambda(W)$, which is an invariant subspace for $A$.  The \textbf{algebraic multiplicity} of $\lambda$, denoted $m_a$ is the dimension of $M$.  Thus, for an isolated eigenvalue of finite algebraic multiplicity, $m_a = \text{dim}(M) < \infty$.  Since $M$ is invariant for $A$, the image satisfies $A_\lambda(M) \subset M$.  In fact, there exists an integer $\ell \geq 1$ such that following chain of inclusions holds:
\begin{equation}
	M \supset A_{\lambda}(M) \supset A_{\lambda}^{2}(M)\supset \dots \supset A_{\lambda}^{\ell}(M) = \{0\},
\end{equation}
each inclusion being proper.
The integer $\ell$ is called the \textbf{ascent} of the eigenvalue.  If $\ell = 1$, then all elements of $M$ are eigenvectors of $A$, that is, $M = \mathcal{N}(A_\lambda)$ is the nullspace of $A_\lambda$.  If $\ell > 1$, then $\mathcal{N}(A_\lambda)$ is a proper subspace of $M$; the elements of $M\setminus\mathcal{N}(A_\lambda)$ are \textbf{generalized eigenvectors} of $A$.  The geometric multiplicity is defined as $m_g = \text{dim}\left(\mathcal{N}(A_\lambda) \right)$; the multiplicities satisfy $m_g \leq m_a$.

\subsubsection{Numerical Approximation of the Spectrum}\label{sec:numerical_spectrum}
Later, we characterize the coercivity constant in~\cref{eq:coercive_form} as an isolated eigenvalue of finite (algebraic) multiplicity.  To establish convergence of corresponding numerical approximations, we utilize results from~\cites{descloux1978spectrala,descloux1978spectralb}, which we now recall.

Fix two continuous sesquilinear forms $\hat{a},b: V\times V\to \mathbb{C}$.  In addition, assume that $b$ is coercive and $\hat{a}(u,u) > 0$ for all $0\neq u \in V$.  Suppose that the bounded linear operator $A$ defined by
\begin{equation}\label{eq:solution_op}
	\hat{a}(u,v) = b(Au,v),\quad \forall u,v\in V,
\end{equation}
has an isolated eigenvalue $0\neq \lambda$ of finite algebraic multiplicity $m_a$.  Consider a family $\left( V_h\right)_{h > 0}$
of finite-dimensional subspaces that satisfy the following \emph{approximability property}~\cite{ern2004theory}:
\begin{equation}\label{eq:approximability}
	\lim_{h\to 0} \inf_{v_h \in V_h}\|v - v_h\|_V = 0,\quad \forall v\in V.
\end{equation}
Denoting by $P_h$, the $V$-orthogonal projector onto $V_h$,~\cref{eq:approximability} is equivalent to
\begin{equation}\label{eq:approximability2}
	\lim_{h\to 0}\|(I - P_h)v\|_V = 0,\quad \forall v\in V,
\end{equation}
where $I:V \to V$ is the identity operator.  The convergence in~\cref{eq:approximability2} is known as strong operator convergence or pointwise convergence; hence $\lim_{h\to 0} P_h v = v$ for all $v\in V$.

\begin{remark}
The approximability property~\cref{eq:approximability} (or~\cref{eq:approximability2}), is a rather mild condition; when $V$ is a Sobolev space, or a related space such as $H(\text{div};\Omega)$, many standard conforming finite element spaces satisfy the approximability property.  A notable exception occurs for $V = H_0(\text{curl};\Omega)\cap H(\text{div};\Omega)$, when $\Omega$ is a non-convex polyhedron in $\mathbb{R}^3$.  In this case, nodal Lagrange finite element spaces do not satisfy the approximability property~\cite{ern2004theory}.
\end{remark}

In addition, suppose that the sequence of linear operators $A_h: V_h \to V_h$ defined by:
\begin{equation}\label{eq:discrete_solution_op}
	\hat{a}(u_h,v_h) = b(A_h u_h,v_h),\quad \forall u_h,v_h\in V_h,
\end{equation}
satisfies
\begin{equation}\label{eq:discrete_norm_convergence}
	\lim_{h\to 0}\sup_{0\neq u_h \in V_h}\frac{\|(A - A_h)u_h\|_V}{\|u_h\|_V} = 0,
\end{equation}
i.e., the sequence $A_h$ converges to $A$ in norm when it is restricted to $V_h$.

Under these assumptions, it is shown in~\cite{descloux1978spectrala} that for $h$ sufficiently small, there are exactly $m_a$ discrete eigenvalues $\lambda^{(1)}_h,\dots, \lambda^{(m_a)}_h$ of $A_h$ inside the set $U$ from~\cref{eq:spectral_projection}, and they converge to $\lambda$ as $h\to 0$.

For the operator $A$ defined in~\cref{eq:solution_op}, there exists an associated bounded linear operator $A^\dagger: V\to V$ defined by~\cite{descloux1978spectralb}:
\begin{equation}\label{eq:b_adjoint}
	b(Au,v) = b(u, A^\dagger v),\quad \forall u,v\in V.
\end{equation}
If $b(u,v) = \overline{b(v,u)}$, i.e., $b$ is an inner product on $V$, this is conceptually identical to the adjoint defined in~\cref{eq:hilbert_adjoint}.  However, the operator $A^\dagger$ is a well-defined bounded linear operator even if $b$ does not define an inner product.  The complex conjugate $\overline{\lambda}$ is an isolated eigenvalue of finite multiplicity $m_a$ for $A^\dagger$, and can be separated from the rest of the spectrum by an open set $U'$~\cite{descloux1978spectralb}.  In analogy with~\cref{eq:spectral_projection}, let
\begin{equation}
	E_\lambda^\dagger = -\frac{1}{2\pi i} \int_{\partial U'} (A^\dagger  - zI)^{-1}\dif z
\end{equation}
be the corresponding spectral projection for $A^\dagger$.  Finally, define the quantities
\begin{subequations}\label{eq:gamma}
	\begin{align}
		\gamma_h &= \sup_{\substack{u \in E_\lambda(V) \\ \|u\|_V = 1}}\inf_{v_h \in V_h}\|u - v_h\|_V\\
		\gamma_h^\dagger &= \sup_{\substack{u \in E_\lambda^\dagger (V) \\ \|u\|_V = 1}}\inf_{v_h \in V_h}\|u - v_h\|_V
	\end{align}
\end{subequations}

Then, the discrete eigenvalues of $A_h$ converge to $\lambda$ as follows (Theorem 3 of~\cite{descloux1978spectralb}):
\begin{subequations}\label{eq:eig_convergence}
	\begin{align}
		\max_{i=1,\dots,m_a} |\lambda - \lambda^{(i)}_h| &\leq C_1 \left(\gamma_h \gamma_h^\dagger\right)^{1/\ell}\\
		\min_{i=1,\dots,m_a} |\lambda - \lambda^{(i)}_h| &\leq C_2 \left(\gamma_h \gamma_h^\dagger\right)^{m_g/m_a},
	\end{align}
\end{subequations}
where $\ell$ is the ascent of $\lambda$ and $m_g$ is the geometric multiplicity.

\section{Characterization of the Coercivity Constant}\label{sec:coercivity}

In this section, we consider~\cref{eq:coercive_form} and characterize the coercivity constant $\alpha$ of a coercive bilinear form $a(u,v)$ through the Rayleigh quotient of a bounded and self-adjoint operator.
Let $X$ and $Y$ be pivot spaces, and let $V$ be a Hilbert space that is continuously and densely embedded in $X$, i.e. $X,V,V'$ form a Gelfand triple.  Let $\LL, \M: V\to Y$ be bounded linear operators satisfying
\begin{equation}
(\LL u,\M u)_Y \geq \alpha \|u\|_V^2, \quad \forall u\in V.
\end{equation}
From the boundedness of $\LL$ and $\M$ it follows that $\left|(\LL u,\M v)_Y\right| \leq \beta\|u\|_V \|v\|_V$, where $\beta = \|\LL\| \|\M\|$.
Thus, the bilinear form
\begin{equation}
a(u,v) \coloneqq (\LL u, \M v)_Y
\end{equation}
is continuous and coercive.

Since $\LL$ and $\M$ are bounded, and $Y$ is a pivot space, the adjoints $\LL'$ and $\M'$ exist as continuous linear operators from $Y \to V'$ satisfying (see~\cref{eq:banach_adjoint_pivot})
\begin{subequations}\label{eq:operator_adjoints}
\begin{align}
	&\left\langle \LL'g, u \right\rangle_{V'\times V} = (g,\LL u)_Y,\\
	&\left\langle \M'g, u \right\rangle_{V'\times V} = (g,\M u)_Y,
\end{align}
\end{subequations}
for every $g\in Y$ and $u\in V$.  Since $\LL'$ and $\M'$ are linear and continuous, then so is the operator $A: V \to V'$ defined by
\begin{equation}\label{eq:definition_A}
	A \coloneqq \frac{1}{2}\left(\M'\LL + \LL'\M \right).
\end{equation}
From~\cref{eq:operator_adjoints}, it follows that for $u,v\in V$:
\begin{align}\label{eq:A_selfadjoint}
\begin{split}
\left\langle Au,v  \right\rangle	_{V'\times V}
&= \frac{1}{2}\left\langle \M'\LL u,v  \right\rangle_{V'\times V} + \frac{1}{2}\left\langle \LL'\M u,v  \right\rangle_{V'\times V}\\
&= \frac{1}{2}(\LL u, \M v)_Y + \frac{1}{2}(\M u,\LL v)_Y\\
&= \frac{1}{2}(\LL v, \M u)_Y + \frac{1}{2}(\M v, \LL u)_Y\\
&= \left\langle Av,u  \right\rangle	_{V'\times V}
\end{split}
\end{align}
Also, for $v = u$, we have
\begin{equation}\label{eq:A_bilinear}
	\left\langle Au,u  \right\rangle	_{V'\times V} = (\LL u, \M u)_Y = a(u,u).
\end{equation}
Let the operator $K: V \to V$ be defined as
\begin{equation}\label{eq:K_definition}
	K\coloneqq A^{-1}R_V,
\end{equation}
where $R_V$ is the Riesz map on $V$, given by \cref{eq:riesz_map}.  We now proceed to show that the coercivity constant $\alpha$ is characterized as the infimum of the Rayleigh quotient of $K^{-1}: V \to V$. We begin by showing that $A$ is invertible, so that $K$ is well-defined, and later establish the connection between the coercivity constant and the Rayleigh quotient of $K^{-1}$.
\begin{theorem}\label{thm:A_bijection}
	The operator $A$ defined by~\cref{eq:definition_A} is a bijection from $V \to V'$.
\end{theorem}
\begin{proof}
	Let $0\neq u\in V$.  By the coercivity of $a(\cdot,\cdot)$, we have:
	\begin{align}\label{eq:injective_A}
	\begin{split}
		\|Au\|_{V'} = \sup_{0\neq v}\frac{|\langle Au, v\rangle|}{\|v\|_V} &\geq \frac{|\langle Au, u\rangle|}{\|u\|_V}\\
		& = \frac{|a(u,u)|}{\|u\|_V} \geq \frac{\alpha\|u\|_V^2}{\|u\|_V} = \alpha \|u\|_V.
	\end{split}
    \end{align}
Thus, if $0\neq u$, then $Au \neq 0$, and $A$ is injective.
Since $V'$ is a Hilbert space, it follows that
\begin{equation}
	V' = \overline{A(V)} \oplus \overline{A(V)}^\perp,
\end{equation}
i.e., $V'$ is the direct sum of the closure of $A(V)$ and its orthogonal complement under the inner product~\cref{eq:dual_inner_product}.

Let $\varphi \in \overline{A(V)}^\perp$.  Then $\varphi \perp A(V)$, so for any $u\in V$, $(\varphi, Au)_{V'} = 0$.  Let $R_V: V\to V'$ be the Riesz map.  Then for any $u \in V$, using the definition of the inner product on $V'$ (see~\cref{eq:dual_inner_product}), and~\cref{eq:A_selfadjoint}, we obtain
\begin{align}
\begin{split}
	0 = (\varphi, Au)_{V'} &= (R_V^{-1}\varphi, R_V^{-1}Au)_V =(R_V^{-1}Au,R_V^{-1}\varphi)_V\\
	&= \left\langle Au, R_V^{-1}\varphi\right\rangle_{V'\times V} =  \left\langle AR_V^{-1}\varphi, u\right\rangle_{V'\times V},
	\end{split}
\end{align}
so $AR_V^{-1}\varphi = 0$.  By injectivity, it follows that $R_V^{-1}\varphi = 0$.  Since the Riesz map is an isometric isomorphism, it follows that $\varphi = 0$.  Thus, $\overline{A(V)}^\perp = \{0\}$ and $V' = \overline{A(V)}$.

Let $\varphi\in \overline{A(V)}$.  Then there exists a sequence $\{u_n\} \subset V$ such that $\varphi = \lim\limits_{n\to\infty}A u_n$.  In particular, $\{Au_n\}$ is a Cauchy sequence in $V'$.  By~\cref{eq:injective_A}, we have $\|Au_n - Au_m\|_{V'} \geq \alpha\|u_n - u_m\|_V$, showing that $\{u_n\}$ is Cauchy in $V$.  Since $V$ is complete, $u_n\to u\in V$, and by continuity of $A$, we have
\begin{equation*}
\varphi = \lim_{n}Au_n = Au.
\end{equation*}
So $\varphi \in A(V)$.  Thus, $A(V) = \overline{A(V)} = V'$.
\end{proof}
Since $V$ and $V'$ are complete, the open mapping theorem~\cite{kreyszig1978introductory} shows that $A^{-1}:V'\to V$ is continuous.  Thus, it follows that the operator $K \coloneqq A^{-1}R_V$ is a well-defined, continuous linear operator mapping $V$ onto itself, with bounded inverse $K^{-1}$.  \Cref{thm:A_bijection} should be compared to the Lax-Milgram theorem for for coercive bilinear forms~\cite{evans2010partial}; the difference here is that we consider the symmetric part of a bilinear form, and an operator between $V$ and its dual.
\begin{theorem}\label{thm:K_symm}
The operator
	$K = A^{-1}R_V$ is self-adjoint and the largest positive constant $\alpha$ satisfying~\cref{eq:coercive_form} is the infimum of the Rayleigh quotient of $K^{-1}$. That is,
	\begin{equation}\label{eq:coercivity_K}
		\alpha \coloneqq \inf_{0\neq u}\frac{a(u,u)}{\|u\|_V^2} = \inf_{0\neq u}\frac{(K^{-1}u,u)_V}{(u,u)_V}.
	\end{equation}
\end{theorem}
\begin{proof}
	If $u,v \in V$, then using~\cref{eq:A_selfadjoint}, we obtain
	\begin{align}
    \begin{split}\label{eq:selfadjoint}
		(Ku,v)_V &= (v,Ku)_V = \langle R_V v, Ku\rangle_{V'\times V} = \langle AA^{-1}R_V v, Ku\rangle_{V'\times V}\\
		&= \langle AKv, Ku\rangle_{V'\times V} = \langle AKu,Kv\rangle_{V'\times V} = \langle R_V u, Kv\rangle_{V'\times V}\\
		&= (u,Kv)_V.
	\end{split}
	\end{align}
  Thus, $K$ is self-adjoint, and so is its inverse $K^{-1} = R_V^{-1}A$.  By~\cref{eq:A_bilinear}, we have
\begin{equation}
(K^{-1}u,u)_V = (R_V^{-1}Au,u)_V = \langle Au, u\rangle_{V'\times V} = a(u,u),
\end{equation}
so that~\cref{eq:coercivity_K} follows.
\end{proof}

It is well-known from the spectral theory of bounded, self-adjoint operators that the infimum of the Rayleigh quotient is a spectral value of the corresponding operator (see Theorem 9.2--3 in~\cite{kreyszig1978introductory}).
The coercivity constant is thus the smallest spectral value of the operator $K^{-1}$.  However, it does not follow that $\alpha$ is an eigenvalue; it may belong to $\sigma(K^{-1})\setminus\sigma_p(K^{-1})$, which presents numerical difficulties in its approximation.

However, with appropriate assumptions on the spaces, the operator $K$ can be characterized as a perturbation of a compact operator, and using this relationship, we show that the non-unit spectral values of $K$ are indeed eigenvalues related to the spectrum of a compact operator.  In this case, the coercivity constant is the smallest eigenvalue of $K^{-1}$, and can be approximated numerically with a high rate of convergence.

\begin{remark}\label{rem:complex_hilbert_space}
The spectral theory of linear operators is fully developed only in the case of complex Hilbert spaces (or complex Banach spaces); cf.~\cref{sec:spectral_theory}.  However, when working within the framework of a real Hilbert $H$, a satisfactory spectral theory can be recovered by viewing $H$ as a subspace of its \emph{complexification}, which is a complex Hilbert space.  See~\cites{MR1115240,MR2344656} for more details.  However, since the operator $K = A^{-1}R_V$ is self-adjoint, its spectrum is real, and nothing is lost by limiting our attention to real Hilbert spaces;~\cites{MR1115240,boffi2010finite}.
\end{remark}

\section{Numerical Approximation of the Coercivity Constant}\label{sec:numeric_convergence}
In~\cref{sec:apps}, we analyze the self-adjoint operator $K$ in~\cref{eq:K_definition} for several variationally posed differential equations, both ordinary (ODEs) and partial (PDEs).
For those examples, we show that operator $K$ has the form
\begin{equation}\label{eq:I_C}
	Ku = (I + C)u,
\end{equation}
where $I:V\to V$ is the identity, and $C: V\to V$ is a compact operator.  We analyze the coercivity constant of~\cref{eq:I_C} in this section, deriving convergence rates for discrete approximations.

\subsection{Representation of the Coercivity Constant}

Since $K$ is self-adjoint, it follows that $C$ is also self-adjoint.  The spectral theory of~\cref{sec:spectral_theory} is simplified when considering compact self-adjoint operators~\cite{conway2019course}.  In particular, the spectrum is a bounded subset of $\mathbb{R}$, and every non-zero $\lambda \in \sigma(C)$ is an eigenvalue.  Furthermore, each non-zero eigenvalue is isolated, has finite algebraic multiplicity, and ascent $\ell = 1$~\cite{boffi2010finite}.  In particular, the algebraic and geometric multiplicities coincide; $m_a = m_g$.

The spectral projection~\cref{eq:spectral_projection} corresponding to each non-zero eigenvalue also simplifies; if $0\neq \lambda\in \sigma(C)$ has multiplicity $m$, there exists an orthonormal set $\{e^{(\lambda)}_1,\dots,e^{(\lambda)}_m\}\subset V$ such that
\begin{equation}
	E_\lambda u = \sum_{k=1}^m \left(u,e^{(\lambda)}_k\right)_V e^{(\lambda)}_k,\quad\forall u\in V.
\end{equation}

Furthermore, the operator $C$ has the representation $C = \sum\limits_{0\neq \lambda \in \sigma(C)} \lambda E_\lambda$.  Thus, counting with multiplicity, there exists a sequence $\{\lambda_1,\lambda_2,\dots\}$ of eigenvalues $\lambda_n \in \mathbb{R}\setminus \{0\}$ with $\lim_{n\to\infty}\lambda_n = 0$, and an orthonormal basis $\{e_1, e_2,\dots\}$ for $\mathcal{N}(C)^\perp$, the orthogonal complement of the null-space of $C$, such that
\begin{equation}\label{eq:series_C}
	Cu = Ku - u = \sum_{n=1}^\infty \lambda_n\ (u, e_n)_V\ e_n,\quad \forall u\in V.
\end{equation}

If $V$ is separable, then there also exists an orthonormal basis $\{\hat{e}_n\}_{n=1}^\infty$ for the null space $\mathcal{N}(C)$.  In the context of differential equations, $V$ is a Sobolev space or a related space such as $H(\text{div};\Omega)$, which are separable. 
In the case of a separable space $V$, every $u\in V$ has the expansion
\begin{equation}\label{eq:series_u}
	u = \sum_{n = 1}^\infty (u,\hat{e}_n)_V\ \hat{e}_n + \sum_{n=1}^\infty (u,e_n)_V\ e_n.
\end{equation}
Combining~\cref{eq:series_C} and~\cref{eq:series_u}, we obtain
\begin{equation}\label{eq:series_K}
	Ku = \sum_{n = 1}^\infty (u,\hat{e}_n)_V\ \hat{e}_n + \sum_{n=1}^\infty (1 + \lambda_n)(u,e_n)_V\ e_n,
\end{equation}
from which it follows that
\begin{equation}\label{eq:series_inv_K}
	K^{-1}u = \sum_{n = 1}^\infty (u,\hat{e}_n)_V\ \hat{e}_n + \sum_{n=1}^\infty \frac{1}{1 + \lambda_n}(u,e_n)_V\ e_n.
\end{equation}

Taking the inner product of~\cref{eq:series_inv_K} with $u$, results in
\begin{equation}\label{eq:Kinv_inner}
	(K^{-1}u, u)_V = \sum_{n = 1}^\infty \left|(u,\hat{e}_n)_V\right|^2 + \sum_{n = 1}^\infty \frac{1}{1 + \lambda_n}\left|(u,e_n)_V\right|^2.
\end{equation}

Recalling the relationship between the coercivity constant and $K^{-1}$ in~\cref{eq:coercivity_K}, we have $(K^{-1}u,u)_V /\|u\|_V^2 \geq \alpha > 0$.  Comparing with~\cref{eq:Kinv_inner}, it follows that $(1 + \lambda_n)^{-1}
> 0$, and thus $\lambda_n > -1$;  furthermore, the coercivity constant satisfies
\begin{equation}\label{eq:alpha_min}
	\alpha = \min\left\{ 1,\ \inf_{n} \frac{1}{1 + \lambda_n} \right\}.
\end{equation}

Since $\lambda_n \to 0$ as $n\to\infty$,~\cref{eq:alpha_min} simplifies to
\begin{equation}\label{eq:alpha_inf}
	\alpha = \inf_{n} \frac{1}{1 + \lambda_n}.
\end{equation}

If at least one eigenvalue $\lambda_n > 0$, then the coercivity constant is related to the largest eigenvalue of the operator $C$, which is an isolated eigenvalue of finite multiplicity.  On the other hand, if $\lambda_n < 0$ for all $n$, then $\alpha = 1$ and complications arise.  If $\mathcal{N}(C) = \{0\}$, then $\alpha$ is not an eigenvalue of $K^{-1}$ (and all sums involving $\hat{e}_n$ above are zero).  If $\mathcal{N}(C) \neq \{0\}$, then $\alpha$ is an eigenvalue, but is not isolated, and may have infinite multiplicity.  In the next section, we assume that $\lambda_n > 0$ for at least one $n\in \mathbb{N}$.

\subsection{Rate of Convergence}

We now consider the convergence of numerical approximations to the coercivity constant.  In order to derive a finite-dimensional variational problem, we need the following lemma:
\begin{lemma}\label{lemma:K_variational}Let $R_V$, $A$, and $K$ be defined as in~\cref{eq:riesz_map},~\cref{eq:definition_A}, and~\cref{eq:K_definition}, respectively.  Then for $u \in V$, the evaluation of $w = Ku$ is equivalent to the problem: find $w \in V$ such that
\begin{equation}\label{eq:variational_def_K}
	\frac{1}{2}(\LL w, \M v)_Y + \frac{1}{2}(\M w, \LL v)_Y = (u,v)_V,\quad \forall v\in V.
\end{equation}
\end{lemma}
\begin{proof} From~\cref{thm:A_bijection}, the operator $K: V\to V$ is well-defined.  Since $K = A^{-1}R_V$, $w = Ku$ if and only if $Aw = R_V u \in V'$.  Thus, for any $v\in V$, we have
\begin{equation}\label{eq:Ru_Aw}
	\langle R_V u, v\rangle_{V'\times V} = \langle Aw,v\rangle_{V'\times V}.
\end{equation}
By the definition of the Riesz map~\cref{eq:riesz_map}, the left hand side of~\cref{eq:Ru_Aw} equals the inner product $(u,v)_V$.

By the equalities in~\cref{eq:A_selfadjoint}, it follows that the right hand side of~\cref{eq:Ru_Aw} equals $\frac{1}{2}(\LL w, \M v)_Y + \frac{1}{2}(\M w, \LL v)_Y$.  That is,~\eqref{eq:variational_def_K} holds.

\end{proof}
It follows from~\cref{lemma:K_variational} that
\begin{equation}\label{eq:Kinv_var}
	(K^{-1}w, v)_V = \frac{1}{2}(\LL w, \M v)_Y + \frac{1}{2}(\M w, \LL v)_Y,\quad\forall w,v\in V.
\end{equation}
We now consider a family of finite-dimensional subspaces $(V_h)_h$ that satisfy the approximability property~\cref{eq:approximability} (or~\cref{eq:approximability2}).
We define an analogous discrete operator $T_h: V_h \to V_h$ by
\begin{equation}\label{eq:Th_var}
	(T_h w_h, v_h)_V = \frac{1}{2}(\LL w_h, \M v_h)_Y + \frac{1}{2}(\M w_h, \LL v_h)_Y,\quad\forall w_h,v_h\in V_h.
\end{equation}
From~\cref{eq:Kinv_var} and~\cref{eq:Th_var}, it follows that for $w_h, v_h \in V_h$,
\begin{equation}
	(T_h w_h, v_h)_V = (K^{-1}w_h, v_h)_V = (P_h K^{-1}w_h, v_h)_V,\quad \forall w_h,v_h\in V_h.
\end{equation}

Thus, $T_h = P_h K^{-1}\vert_{V_h}$, and from~\cref{eq:approximability2}, we see that $T_h \to K^{-1}\vert_{V_h}$ pointwise. Given that $V_h$ is finite-dimensional, pointwise (strong operator) convergence implies convergence in norm:
\begin{equation}\label{eq:operator_convergence}
	\lim_{h\to 0}\sup_{0\neq u_h\in V_h} \frac{\|(K^{-1} - T_h)u_h \|}{\|u_h\|_V} = 0.
\end{equation}

Returning to~\cref{eq:alpha_min} and the discussion which precedes it, we assume that there exists at least one eigenvalue $\lambda_n > 0$.  Then the coercivity constant is an isolated eigenvalue of $K^{-1}$ with finite algebraic multiplicity $m$, which coincides with the geometric multiplicity.  Defining the bilinear forms
\begin{subequations}\label{eq:cc_bilinear_forms}
	\begin{align}
		\hat{a}(u,v) &= \frac{1}{2}(\LL w, \M v)_Y + \frac{1}{2}(\M w, \LL v)_Y,\\
		b(u,v) &= (u,v)_V,
	\end{align}
\end{subequations}
it follows from~\cref{lemma:K_variational} that
\begin{subequations}
	\begin{align}
		\hat{a}(u,v) &= b(K^{-1}u,v),\quad \forall u,v\in V,\\
		\hat{a}(u_h,v_h) &= b(T_h u_h,v_h),\quad \forall u_h, v_h \in V_h.
	\end{align}
\end{subequations}
We note further that $b$ is coercive and for $0\neq u$, $\hat{a}(u,u) \geq \alpha \|u\|_V^2 > 0$.

Thus, with $\hat{a}$ and $b$ defined by~\cref{eq:cc_bilinear_forms}, the framework of~\cref{sec:numerical_spectrum} applies; cf.~\cref{eq:solution_op}.  If the family $(V_h)_h$ satisfies~\cref{eq:approximability}, we have shown that discrete norm convergence~\cref{eq:discrete_norm_convergence} holds, and convergence rates of numerical approximations to the coercivity constant can be established through~\cref{eq:gamma} and~\cref{eq:eig_convergence}.

Since $K^{-1}$ is self-adjoint (cf.~\cref{thm:K_symm}), we have $b(K^{-1}u,v) = (K^{-1}u,v)_V = (u,K^{-1}v)_V = b(u,K^{-1}v)$ and so by~\cref{eq:b_adjoint} and~\cref{eq:gamma}, we have
\begin{equation}
	\gamma_h^\dagger = \gamma_h = \sup_{\substack{u\in E \\ \|u\|_V=1}} \inf_{v_h \in V_h}\|u - v_h\|_V,
\end{equation}
where $E$ is the eigenspace corresponding to the coercivity constant.

We recall that for this eigenvalue, the ascent $\ell = 1$ and the algebraic and geometric multiplicities coincide.  Thus the right-hand sides of~\cref{eq:eig_convergence} are of the same order and we simply have
\begin{equation}\label{eq:max_eig_bound}
	\min_{i=1,\dots,m} |\alpha - \mu_i^{(h)}|\leq \max_{i=1,\dots,m} |\alpha - \mu_i^{(h)}| \leq c\gamma_h^2,
\end{equation}
where $\mu_1^{(h)}, \dots, \mu_m^{(h)}$ are the eigenvalues of $T_h$ converging to $\alpha$.
In addition, since the bilinear form $\langle Au, v\rangle_{V'\times V} = \frac{1}{2}(\LL u, \M v)_Y + \frac{1}{2}(\M u, \LL v)_Y$ is symmetric and coercive, the discrete eigenvalues satisfy a monotonicity property~\cite{boffi2010finite}:
\begin{equation}\label{eq:eig_monotonicity}
	\alpha \leq \mu_i^{(h)},\quad i=1,\dots,m.
\end{equation}
Thus, combining~\cref{eq:max_eig_bound} and~\cref{eq:eig_monotonicity}, with $\alpha_h := \min_{i}\mu_i^{(h)}$, we obtain
\begin{equation}\label{eq:eigenvalue_convergence}
	\alpha \leq \alpha_h \leq \alpha + c\sup_{\substack{u\in E \\ \|u\|_V=1}}\inf_{v_h \in V_h} \|u - v_h\|_V^2.
\end{equation}
Equation~\cref{eq:eigenvalue_convergence} shows that the coercivity constant is approximated from above, with convergence rate that is double the rate at which eigenfunctions are approximated by $V_h$.

\section{Applications and Derivations}\label{sec:apps}
We derive an expression for the operator $K = A^{-1}R_V$ for several differential equations; in particular, we establish a connection with a compact operator $C$ as in~\cref{eq:I_C}.

\subsection{Computational Aspects}

Evaluating the map $K$ in~\cref{eq:K_definition} does not require an explicit expression for the Riesz map $R_V$, the adjoints $\LL'$ and $\M'$, nor $A^{-1}$, which is an integral operator in the context of differential equations.  In fact, from~\cref{lemma:K_variational} if $w = Ku$, then $w$ is the solution of the variational problem~\cref{eq:variational_def_K}.

For the scalar differential equations considered, we use~\cref{eq:variational_def_K} to solve for $Ku - u = w - u$, and show that the map $u\mapsto w - u$ defines a compact operator from $V \to V$.  In~\cref{sec:ode-first-order}, compactness is shown by demonstrating that the range of this mapping is finite-dimensional (thus compact, see Section 2.8 of~\cite{atkinson2005theoretical}).  In~\cref{sec:ode_bvp} and~\cref{sec:example_2}, we demonstrate the range consists of sufficiently smooth functions, so that the compact embedding results of the Rellich-Kondrachov theorem can be applied (see Theorem II.1.9 in~\cite{braess2007finite}).

For systems of differential equations with unknowns $u_1, u_2, \dots, u_n$, the linearity of $K$ is exploited and we write
\begin{align}
\begin{split}
K\vu &= K[u_1,u_2,\dots,u_n]^T\\
&= K[u_1,0,\dots, 0]^T + K[0, u_2,\dots, 0]^T + \dots K[0, 0, \dots, u_n]^T.
\end{split}
\end{align}
In this case, $K[0,\dots,u_k,\dots, 0]^T = \vw_k$ is equivalent to the variational problem
\begin{equation}
	\frac{1}{2}(\LL \vw_k, \M \vv)_Y + \frac{1}{2}(\M \vw_k, \LL \vv)_Y = ([0,\dots,u_k,\dots,0]^T,\vv)_V,\quad \forall \vv\in V.
\end{equation}

Then the analysis proceeds by considering $K\vu - \vu$, and showing sufficient smoothness in order to apply the Rellich-Kondrachov theorem.

\subsection{Application: Least-Squares IVP}\label{sec:ode-first-order}

Consider the basic 1D model problem
\begin{subequations}\label{eq:example_1}
\begin{align}
u'(x) - u(x) &= f(x), \quad x\in(0,1)\\
u(0) &= 0,
\end{align}
\end{subequations}
for $f \in L^2(0,1)$.  The solution space is $V = \{ u\in H^1(0,1) \,:\, u(0) = 0\} $, and a least-squares variational equation takes the form: find $u \in V$ such that
\begin{equation}
	(u' - u, v' - v)_0 = (f,v' - v)_0\quad \forall v\in V,
\end{equation}
where $(\cdot, \cdot)_0$ denotes the $L^2$ inner product.  This corresponds to $X = Y = L^2(0,1)$, with operators of the form $\LL = \M = \frac{\dif}{\dif x} - I_X$, where $I_X$ is the identity operator on $X$ (restricted to the space $V$).
Since $\LL = \M$, the corresponding variational equation~\cref{eq:variational_def_K} for $w = Ku$ is: find $w \in V$ such that
\begin{equation}\label{eq:K_example_1}
	(w' - w, v' - v)_0 = (u,v)_0 + (u',v')_0,\quad \forall v \in V.
\end{equation}

Consider $v \in C^\infty_0(0,1) \subset V$. Then by definition of distributional derivatives we have that
\begin{equation}\label{eq:ex1_wu}
	(w,v)_0 - \langle w'', v\rangle = (u,v)_0 - \langle u'', v\rangle, \quad \forall v \in C^\infty_0(0,1),
\end{equation}
where the duality pairing $\langle\cdot,\cdot\rangle$ is between $C_0^\infty(0,1)$ and its dual.

By~\cref{eq:ex1_wu}, it follows that $(w - u)'' = w - u \in V \subset H^1(0,1)$.  It follows that the function $\phi \coloneqq Ku - u = w - u \in H^3(0,1) \cap V$.  That is, it satisfies the differential equation:
\begin{subequations}
\begin{align}
	\phi'' &= \phi \quad \text{in } (0,1),\\
	\phi(0) &= 0.
\end{align}
\end{subequations}
This has the general solution $\phi(x) = \beta \sinh(x)$ for some constant $\beta$.

To determine $\beta$, consider~\cref{eq:K_example_1} for any $v\in V$ with $v(1) \neq 0$.  Using the definition of $\phi = w - u$, we obtain
\begin{equation}
	(\phi,v)_0 + (\phi',v')_0 = (w',v)_0 + (w,v')_0.
\end{equation}
Performing integration by parts, and using the fact that $\phi - \phi'' = 0$, we obtain
\begin{equation}
	v(1)\phi'(1) = v(1)w(1) = v(1)\left(u(1) + \phi(1)\right).
\end{equation}
Thus, $\phi'(1) - \phi(1) = u(1)$, leading to $\beta = \me\,u(1)$.  Since $Ku = u + \phi$, we thus have
\begin{equation}
	Ku = u + \me\,u(1)\sinh(x), \quad \forall u \in V.
\end{equation}
The map $C: u\mapsto \me\,u(1)\sinh(x)$ has finite-dimensional range, namely $\spn\{\sinh\}$.  Thus, it is compact~\cite{kreyszig1978introductory},
 and $K = I + C$, where $C: V\to V$ is a compact operator, and $I: V\to V$ is the identity.

Next we take a closer look at the eigenvalues and eigenvectors of $K$.
Since $C$ is compact, its spectrum contains $0$, and an at-most countable set of eigenvalues~\cite{kreyszig1978introductory}.  Moreover, if $u \in H_0^1(0,1) \subset V$, we have $Cu = \me\,u(1)\sinh(x) = 0$ (since $u(1)=0$), establishing that $0$ is an eigenvalue of $C$ with infinite-dimensional eigenspace $H_0^1(0,1)$.

If $\lambda\neq 0$ is an eigenvalue of $C$, then
\begin{equation}
	\lambda u(x) = \me\,u(1)\sinh(x),
\end{equation}
for $u \neq 0$.  In particular, $u(1) \neq 0$, and $u \in \spn\{\sinh(x)\}$.  Solving for $\lambda$, we find that $\lambda = \me\sinh(1)$.  Thus, the set of eigenvalues for $K = I + C$ is $\{1, 1 + \me\sinh(1)\}$, and
\begin{subequations}
	\begin{align}
		Ku = u &\iff u \in H_0^1(0,1),\label{eq:ex1_eig1}\\
		Ku = (1 + \me\sinh(1))u &\iff u \in \spn\{\sinh(x)\}.\label{eq:ex1_eig2}
	\end{align}
\end{subequations}
Finally, the eigenvalues of $K^{-1}$ are $\sigma(K^{-1}) = \left\{1, \frac{1}{1 + \me\sinh(1)}\right\}$, and the coercivity constant satisfies
\begin{equation}
	\alpha = \inf_{0 \neq u\in V} \frac{(u' - u, u' - u)_0}{(u,u)_0 + (u', u')_0} = \frac{1}{1 + \me\sinh(1)}.
\end{equation}

\subsection{Application: Galerkin Formulation of Advection-Diffusion BVP}\label{sec:ode_bvp}

We next consider an ordinary differential equation that resembles diffusion and advection (in 1D):
\begin{subequations}
\begin{align}
  -u''(x) + u'(x) &= f(x), \quad x\in (0,1),\\
  u(0) = u(1) & = 0,
\end{align}
\end{subequations}
where $f \in L^2(0,1)$.  The solution space is thus $V = H_0^1(0,1)$, we define the following variational problem: find $u \in V$ such that
\begin{equation}
		(u',v + v')_0 = (f,v)_0, \quad \forall v\in H^1_0(0,1).
\end{equation}
Thus, the operators take the form $\LL = \frac{\dif}{\dif x}$, and $\M = I_X + \frac{\dif}{\dif x}$, with $X = Y = L^2(0,1)$ and $I_X: V\to X$ is the identity operator on $X$, restricted to $V$.

If $Ku = w$, then we have
\begin{equation}
	\left\langle R_V u, v\right \rangle_{V'\times V} = \frac{1}{2}\left\langle \M'\LL w,v\right\rangle_{V'\times V}  + \frac{1}{2}\left\langle \LL'\M w,v\right\rangle_{V'\times V} ,
\end{equation}
for all $v\in V$, which leads to
\begin{align}
\begin{split}
  (u,v)_0 + (u',v')_0 & = \frac{1}{2}(w', v + v')_0 + \frac{1}{2}(w + w', v')_0,\\
	                    & = (w',v')_0 + \frac{1}{2}(w',v)_0 + \frac{1}{2}(w,v')_0,
	                    \end{split}
\end{align}
for all $v\in V$.  Then, choosing $v \in C^\infty_0(0,1) \subset V$, we find that
\begin{equation}\label{eq:example_2_phi_ode}
-\phi'' = u,
\end{equation}
with $\phi \coloneqq Ku - u = w - u$.  Since $w,u\in H_0^1(0,1)$, it follows that $\phi \in H^3(0,1)\cap H_0^1(0,1)$.
By the Rellich-Kondrachov theorem (again, see Theorem II.1.9 in~\cite{braess2007finite}), $H^{k+1}(0,1)$ is compactly embedded in $H^{k}(0,1)$ for every non-negative integer $k$. Thus, for any bounded sequence $\{\varphi_n\} \subset H^3(0,1)\cap H_0^1(0,1) \subset H^3(0,1)$, there is a subsequence $n_k$ and $\varphi \in H^2(0,1)$ such that $\varphi_{n_k} \to \varphi$ in $H^2(0,1)$.  Since $\|\varphi_{n_k} - \varphi\|_{H^1} \leq \|\varphi_{n_k} - \varphi\|_{H^2}$, it also follows that $\varphi_{n_k}$ converges to $\varphi$ in $H^1(0,1)$.  Since $H_0^1(0,1)$ is a closed subspace, $\varphi_{n_k} \in H_0^1(0,1)$ implies that $\varphi\in H_0^1(0,1)$.  Thus, $H^3(0,1)\cap H_0^1(0,1)$ is compactly embedded in $V = H_0^1(0,1)$.

The map $u \mapsto Cu =  \phi$ induced by~\cref{eq:example_2_phi_ode} thus maps $V$ into $H^3(0,1)\cap H_0^1(0,1)$, a compactly embedded subspace, and is thus a compact operator.  More precisely, $\phi = w - u = Cu$, where $C$ is the compact solution operator of the ODE~\cref{eq:example_2_phi_ode}.
As a result, $Ku = u + \phi = u + Cu = (I + C)u$.

\subsection{Application: Advection-Diffusion-Reaction PDE}\label{sec:example_2}

We extend the previous example to multiple dimensions with a reaction term:
\begin{align}\label{eq:pde-advec-diff}
\begin{split}
	-\Delta u + \vb\cdot\nabla u + u &= f, \quad \vx \in \Omega,\\
	u &= 0, \quad \vx \in \partial\Omega,
	\end{split}
\end{align}
where $\vb(\vx)$ is a smooth vector field with $\nabla\cdot\vb \leq 2$, and $\Omega$ is a bounded open subset of $\mathbb{R}^n$ that satisfies a cone condition with Lipschitz continuous boundary $\partial\Omega$ (see~\cites{braess2007finite,gilbarg2015elliptic, grisvard2011elliptic}).

In this case the variational equation
\begin{equation}\label{eq:bilinear_form_example_3}
	a(u,v) \coloneqq (\nabla u,\nabla v)_0 + (\vb\cdot\nabla u, v)_0 + (u,v)_0 = (f,v)_0, \quad \forall v\in H_0^1(\Omega),
\end{equation}
yields a coercive bilinear form when $\vb$ satisfies the properties above.

To associate operators with the weak form above, consider
\begin{equation}
	\LL \coloneqq \begin{pmatrix}
		\nabla \\
		I_X + \vb\cdot \nabla
	\end{pmatrix},
\end{equation}
with $I_X: H_0^1(\Omega) \to L^2(\Omega)$ the restriction of the identity on $L^2(\Omega)$; $\LL$ is a mapping from $H_0^1(\Omega) \to \left[L^2(\Omega)\right]^n \times L^2(\Omega)$.
Similarly the operator
\begin{equation}
	\M \coloneqq \begin{pmatrix}
		\nabla \\
		I_X
	\end{pmatrix}
\end{equation}
maps $H_0^1(\Omega) \to \left[L^2(\Omega)\right]^n \times L^2(\Omega)$.  Thus, the pivot spaces for this problem are $X = L^2(\Omega)$ and $Y = \left[L^2(\Omega)\right]^n \times L^2(\Omega)$.  With these operators, the bilinear form in~\cref{eq:bilinear_form_example_3} is written as
\begin{equation}
  a(u,v) = (\LL u, \M v)_0.
\end{equation}
Here we use the same notation in the inner product for the $L^2$ product space.

If $Ku = w$, then
\begin{equation}
	(u,v)_V = \frac{1}{2}\left(\LL w,\M v\right)_0 + \frac{1}{2}\left(\M w,\LL v\right)_0,
\end{equation}
for all $v\in V$.  From this, we see that
\begin{equation}\label{eq:K_example_3}
 (u,v)_0 + (\nabla u,\nabla v)_0 = (w,v)_0 + (\nabla w,\nabla v)_0
	+ \frac{1}{2}(\vb\cdot\nabla w, v)_0 + \frac{1}{2}(w, \vb\cdot\nabla v)_0,
\end{equation}
for all $v\in V$.
Choosing $v\in C_0^\infty(\Omega)$, integration by parts shows that for $\phi = w - u$,
\begin{equation}
	-\Delta \phi + \phi = \frac{1}{2}(\nabla\cdot\vb)w.
\end{equation}
Adding the term $\frac{1}{2}(\nabla\cdot\vb)u$ to both sides, we obtain:
\begin{equation}\label{eq:example_3_phi_pde}
	-\Delta \phi + (1 - \frac{1}{2}\nabla\cdot\vb)\phi = \frac{1}{2}(\nabla\cdot\vb)u.
\end{equation}
Since $\nabla\cdot\vb \leq 2$, the coefficient $1 - \frac{1}{2}\nabla\cdot\vb \geq 0$, and~\cref{eq:example_3_phi_pde} corresponds to an elliptic PDE\@.  With our assumptions on the domain $\Omega$, $\phi \in H^3(\Omega) \cap V$, and the Rellich-Kondrachov theorem applies.  As in~\cref{sec:example_2}, the compact embedding of $H^3(\Omega)$ into $H^2(\Omega)$ and the fact that $H_0^1(\Omega)$ is a closed subspace, shows that $Ku = u + \phi$ is of the form $I + C$, with $C$ a compact operator.

\subsection{Application: Poisson Equation, least-squares formulation}\label{sec:poisson_ls}

We next consider a slightly easier PDE, but consider
the first-order reformulation. An equivalent first-order system to the Poisson equation with homogeneous Dirichlet boundary conditions is given by
\begin{subequations}\label{eq:ls-poisson}
\begin{align}
	\vq + \nabla u &= 0, \quad \vx \in \Omega,\\
	\nabla\cdot\vq &= f, \quad \vx \in \Omega,\\
	u &= 0, \quad \vx \in \partial\Omega,
\end{align}
\end{subequations}
and we consider the following operator:
\begin{equation}
	\LL = \begin{pmatrix}
		I_X & \nabla \\
		\nabla\cdot & 0
	\end{pmatrix},
\end{equation}
with $I_X: H(\text{div}) \to \left[L^2(\Omega)\right]^n$ the restriction of the identity operator on $\left[L^2(\Omega)\right]^n$.  The operator $\LL$ maps $V \coloneqq H(\text{div})\times H_0^1 \to \left[L^2(\Omega)\right]^n \times L^2(\Omega)$, with corresponding pivot spaces $X = Y = \left[L^2(\Omega)\right]^n \times L^2(\Omega)$.

A least-squares finite element formulation leads to a symmetric bilinear
form~---~i.e., $\M = \LL$.  Thus, the operator $A = \LL'\LL$.

Notationally, we consider an arbitrary element of $V = H(\text{div})\times H_0^1$ to be $[\vr, v]^T$.
If $K[\vq, u]^T = [\vf, w]^T$ then
\begin{equation}\label{eq:K_example_4}
	\left(\begin{bmatrix} \vq \\ u\end{bmatrix},\begin{bmatrix} \vr \\ v\end{bmatrix}\right)_V = \left(\LL \begin{bmatrix} \vf \\ w\end{bmatrix},\LL \begin{bmatrix} \vr \\ v\end{bmatrix}\right)_0, \quad \forall \begin{bmatrix} \vr \\ v\end{bmatrix}\in V,
\end{equation}
which gives the two equations
\begin{subequations}\label{eq:poisson_A_eq}
\begin{align}
  (\vq, \vr)_0 + (\nabla\cdot \vq, \nabla \cdot\vr)_0 &= (\vf + \nabla w, \vr)_0 + (\nabla\cdot \vf, \nabla \cdot\vr)_0\quad\forall\vr \in H(\text{div})\\
  (u, v)_0 + (\nabla u, \nabla v)_0 &= (\vf + \nabla w, \nabla v)_0\quad \forall v \in H_0^1(\Omega).
\end{align}
\end{subequations}
Since $K$ is a linear map, we can consider the cases $\vq = 0$ and $u = 0$ separately and take the superposition $K[\vq, u]^T = K[0, u]^T + K[\vq, 0]^T$.

\medskip
\subsubsection*{Case I\@. $\vq = 0$}~\\
Problem~\cref{eq:poisson_A_eq} becomes
\begin{subequations}\label{eq:zero_q_hdiv}
\begin{align}
	(\vf + \nabla w, \vr)_0 + (\nabla\cdot \vf, \nabla \cdot\vr)_0 = 0\quad\forall\vr \in H(\text{div})&\\\label{eq:zero_q_h1}
(\vf + \nabla w, \nabla v)_0 = (u, v)_0 + (\nabla u, \nabla v)_0\quad \forall v \in H_0^1(\Omega).
\end{align}
\end{subequations}
As a result, choosing $\vr \in \left[C_0^\infty\right]^n$
in~\cref{eq:zero_q_hdiv} shows that $\nabla\nabla\cdot\vf = \vf + \nabla w \in L^2$,
so that $\nabla\cdot \vf \in H^1$. At the same time, choosing $\vr\in H(\text{div})$ with non-vanishing trace leads us to conclude that $\nabla\cdot\vf$ vanishes on the boundary.
Inserting $\nabla\nabla\cdot\vf = \vf + \nabla w$ into~\cref{eq:zero_q_h1} and choosing $v\in C_0^\infty$, we obtain
\begin{equation}\label{eq:div_f}
	-\Delta(\nabla\cdot\vf) = u - \Delta u \implies \nabla\cdot\vf = u + (-\nabla)^{-1}u.
\end{equation}

Continuing, since $\nabla w = \nabla\nabla\cdot\vf - \vf$, it follows that
\begin{equation}\label{eq:grad_g}
	\nabla w = \nabla u + \nabla(-\Delta)^{-1}u - \vf.
\end{equation}
Taking the divergence of \cref{eq:grad_g} and using \cref{eq:div_f}, we obtain
\begin{align}
\begin{split}
	-\Delta w &= -\Delta u + (-\Delta)(-\Delta)^{-1}u + \nabla\cdot\vf \\
	&= 2u - \Delta u + (-\Delta)^{-1}u.
	\end{split}
\end{align}
Since $w\in H_0^1$, applying the inverse Laplacian leads to
\begin{equation}\label{ex:g_zero_q}
w = u + 2(-\Delta)^{-1}u + (-\Delta)^{-2}u.
\end{equation}
Finally, using the expressions for $w$, $\nabla\cdot \vf$, and $\vf + \nabla w = \nabla\nabla\cdot \vf$, we obtain
\begin{equation}\label{ex:f_zero_q}
	\vf = -\nabla(-\Delta)^{-1}u - \nabla(-\Delta)^{-2}u.
\end{equation}

\medskip
\subsubsection*{Case II\@. $u=0$}~\\
In this scenario, problem~\cref{eq:poisson_A_eq} becomes:
\begin{subequations}\label{eq:zero_u_hdiv}
\begin{align}
  (\vf + \nabla w, \vr)_0 + (\nabla\cdot \vf, \nabla \cdot\vr)_0 &= (\vq, \vr)_0 + (\nabla\cdot \vq, \nabla \cdot\vr)_0\quad\forall\vr \in H(\text{div})\\\label{eq:zero_u_h1}
  (\vf + \nabla w, \nabla v)_0 &= 0\quad \forall v \in H_0^1(\Omega).
\end{align}
\end{subequations}
Then, choosing $v\in C_0^\infty$ in \cref{eq:zero_u_h1} yields
\begin{equation}\label{eq:g_zero_u}
	-\Delta w = \nabla \cdot \vf \implies w = (-\Delta)^{-1}\nabla\cdot\vf,
\end{equation}
so that $\nabla w = \nabla (-\Delta)^{-1}\nabla\cdot\vf$.  Letting $\vr \in \left[C_0^\infty\right]^n$ in~\cref{eq:zero_u_hdiv}, we see that
\begin{equation}\label{eq:div_f_minus_q}
\nabla\nabla\cdot (\vf - \vq)= \vf - \vq	 + \nabla (-\Delta)^{-1}\nabla\cdot\vf \in L^2(\Omega).
\end{equation}
Thus, $\nabla \cdot(\vf - \vq) \in H^1$.  Repeating the computation with $\vr \in H(\text{div})$ with non-vanishing trace, it follows that $\nabla \cdot(\vf - \vq)$ vanishes on the boundary.  Thus, $(-\Delta)^{-1}(-\Delta)\nabla\cdot (\vf - \vq) = \nabla\cdot (\vf - \vq)$.  Applying this after taking the divergence of~\cref{eq:div_f_minus_q} leads to
\begin{equation}\label{eq:div_f_zero_u}
	\nabla\cdot\vf = \nabla\cdot \vq + (-\Delta)^{-1}\nabla\cdot\vq.
\end{equation}
From this and \cref{eq:g_zero_u}, we obtain
\begin{equation}
	w = (-\Delta)^{-1}\nabla\cdot\vq + (-\Delta)^{-2}\nabla\cdot\vq.
\end{equation}
Finally, combining \cref{eq:div_f_minus_q,eq:div_f_zero_u} leads to
\begin{equation}
	\vf = \vq - \nabla(-\Delta)^{-2}\nabla\cdot \vq.
\end{equation}\par

Using these two cases~---~for $\vq = 0$ and $u = 0$~---~it follows that
\begin{align}\label{eq:K_standard}
\begin{split}
K[\vq, u]^T &= \begin{pmatrix}
 \vq - \nabla(-\Delta)^{-2}\nabla\cdot \vq - \nabla(-\Delta)^{-1}u - \nabla(-\Delta)^{-2}u\\
 u + 2(-\Delta)^{-1}u + (-\Delta)^{-2}u + (-\Delta)^{-1}\nabla\cdot\vq + (-\Delta)^{-2}\nabla\cdot\vq
 \end{pmatrix}\\
 &= \begin{bmatrix}
 	\vq\\u
 \end{bmatrix} + \begin{pmatrix}
 	-\nabla (-\Delta)^{-2}\nabla\cdot & -\nabla (-\Delta)^{-1} - \nabla (-\Delta)^{-2}\\
 	(-\Delta)^{-1}\nabla\cdot + (-\Delta)^{-2}\nabla\cdot & 2(-\Delta)^{-1} + (-\Delta)^{-2}
 \end{pmatrix}\begin{bmatrix}
 	\vq\\u
 \end{bmatrix}\\
 &= (I + C)[\vq, u]^T.
 \end{split}
\end{align}

Examining the components of $C$, we observe that the ranges are:
\begin{subequations}\label{eq:range_C}
\begin{align}
	\mathcal{R}(C_{11}) \subset \left[H^3\right]^n,\\
	\mathcal{R}(C_{12}) \subset \left[H^2\right]^n,\\
	\mathcal{R}(C_{21}) \subset H^2\cap H_0^1,\\
	\mathcal{R}(C_{22}) \subset H^3\cap H_0^1.
\end{align}
\end{subequations}

Another application of the Rellich-Kondrachov theorem shows that the operator $C$ is compact.  Indeed, $\left[H^3\right]^n$ and $\left[H^2\right]^n$ are compactly embedded in $\left[H^1\right]^n$, and $\|\cdot\|_{H(\text{div})} \leq C\|\cdot\|_{\left[H^1\right]^n}$ so the ranges of $C_{11}$ and $C_{12}$ are compactly embedded in $H(\text{div})$.  The compact embedding of the ranges of $C_{21}$ and $C_{22}$ follow as in the scalar case.

\subsection{Application: Poisson Equation, rescaled least-squares formulation}\label{sec:poisson-rescaled}

The first-order reformulation of the Poisson equation that we considered in the previous section is not unique.  There, we introduce the variable $\vq \coloneqq -\nabla u$.  If instead we introduce $\vq \coloneqq -2\nabla v$, the first-order system becomes
\begin{subequations}\label{eq:ls-poisson-scaled}
\begin{align}
	\vq + 2\nabla u &= 0, \quad \vx \in \Omega,\\
	\nabla\cdot\vq &= 2f, \quad \vx \in \Omega,\\
	u &= 0, \quad \vx \in \partial\Omega.
\end{align}
\end{subequations}
From this, we consider the operator
\begin{equation}
	\LL = \begin{pmatrix}
		I_X & 2\nabla \\
		\nabla\cdot & 0
	\end{pmatrix},
\end{equation}
with $V = H(\text{div})\times H_0^1$ and the same pivot spaces as in~\cref{sec:poisson_ls}.

Following a similar derivation, we arrive at $K$ of the form
\begin{align}\label{eq:D_C}
\begin{split}
	K(\vq, u)&= \begin{pmatrix}
		\vq \\ \frac{1}{4}u
	\end{pmatrix}\\
           &+ \begin{pmatrix}
 	-\nabla (-\Delta)^{-2}\nabla\cdot & -\frac{1}{2}\nabla (-\Delta)^{-1} - \frac{1}{2}\nabla (-\Delta)^{-2}\\
 	\frac{1}{2}(-\Delta)^{-1}\nabla\cdot + \frac{1}{2}(-\Delta)^{-2}\nabla\cdot & \frac{1}{2}(-\Delta)^{-1} + \frac{1}{4}(-\Delta)^{-2}
 \end{pmatrix}\begin{pmatrix}
 	\vq\\u
 \end{pmatrix}\\
 &= (D + C)(\vq, u),
	\end{split}
\end{align}
where $D = \text{diag}(1, \frac{1}{4})$.
The presence of the diagonal operator precludes the simple characterization of the spectrum as in~\cref{eq:series_K}.  However, as shown in~\cref{sec:ls_poisson}, the convergence of the discrete coercivity does not deteriorate; indeed it behaves as in~\cref{eq:eigenvalue_convergence}.

A heuristic explanation is as follows.  Multiplying~\cref{eq:D_C} on the left by $D^{-1}$ results in
\begin{equation}
KD^{-1} = I + CD^{-1} = I + \widetilde{C}.
\end{equation}
Since $D^{-1}$ is bounded and $C$ is compact, it follows that their product $\widetilde{C} = CD^{-1}$ is also a compact operator.  As a result, the form of~\cref{eq:I_C} is recovered for the operator $KD^{-1}$.

To interpret the operator $KD^{-1}$ in terms of the inner product and bilinear form on $V$, we return to the definition in~\cref{eq:K_definition}.  From this we have
\begin{equation}
	KD^{-1} = A^{-1} R_V D^{-1}.
\end{equation}

The operator $R_V D^{-1}$ constitutes a rescaling of the inner-product and norm of the space $V$, as in
\begin{multline}
	\left\langle R_V D^{-1}(\vq, u), (\vr, v)\right\rangle_{V'\times V} = (D^{-1}(\vq,u), (\vr, v))_V\\
	= (\vq, \vr)_0 + (\nabla\cdot\vq, \nabla\cdot\vr)_0 + 4(u,v)_0 + 4(\nabla u, \nabla v)_0,
\end{multline}
which is an equivalent inner product, and thus leads to an equivalent norm.  The operator $\widetilde{C}$ is symmetric with respect to this new inner product, and so the convergence of discrete eigenvalues are once again governed by the theory in~\cites{descloux1978spectrala,descloux1978spectralb}.  By norm equivalence, we would expect the same order of convergence for the original problem $K = D + C$.

Rigorously justifying this argument necessitates a careful analysis of the spectrum and eigenspaces of the operators $K$ and $KD^{-1}$.  More generally, we can expect the operator $K$ to take the form $K = M + C$, where $M$ is a non-diagonal, invertible operator.  This may be the case, for example, when solving a diffusion equation in an anisotropic medium.  Establishing a relationship between the spectrum of $C$ and $K$ in this case is unresolved, and is the subject of future work.

\section{Numerical Results}\label{sec:numerics}

In~\cref{sec:apps} several relationships for coercivity constants were derived for
various differential equations.  Next, we revisit several of these cases and highlight
the numerical accuracy of the discrete coercivity constant.

\subsection{Least-Squares Initial Value Problem}

In \cref{sec:ode-first-order} the least-squares variational formulation of an
initial value problem~\cref{eq:example_1} is shown to have a coercivity
constant satisfying
\begin{equation}
  \alpha = \inf_{0 \neq u\in V} \frac{(u' - u, u' - u)_0}{(u,u)_0 +
  (u', u')_0} = \frac{1}{1 + e\sinh(1)} = 1 - \tanh(1),
\end{equation}
with corresponding eigenspace of a single function,
\begin{equation}
  \spn\{\sinh(x)\}\subset V = \left\{ v\in H^1(0,1) : v(0) = 0 \right\}.
\end{equation}

In \cref{fig:ode-eig-linear}, the convergence of the discrete coercivity
constant to the coercivity constant of the continuous problem is shown as a
function of the mesh size $h$. The results show both linear and quadratic polynomial
elements, and we
see the convergence rate of the coercivity constant is
twice the rate of convergence for approximation of general functions in $H^1 \supset V$,
resulting in $\mathcal{O}(h^2)$ and $\mathcal{O}(h^4)$ for linears and quadratics, as expected (cf.,~\cite{boffi2010finite}).
\begin{figure}[!ht]
	\centering
	\includegraphics{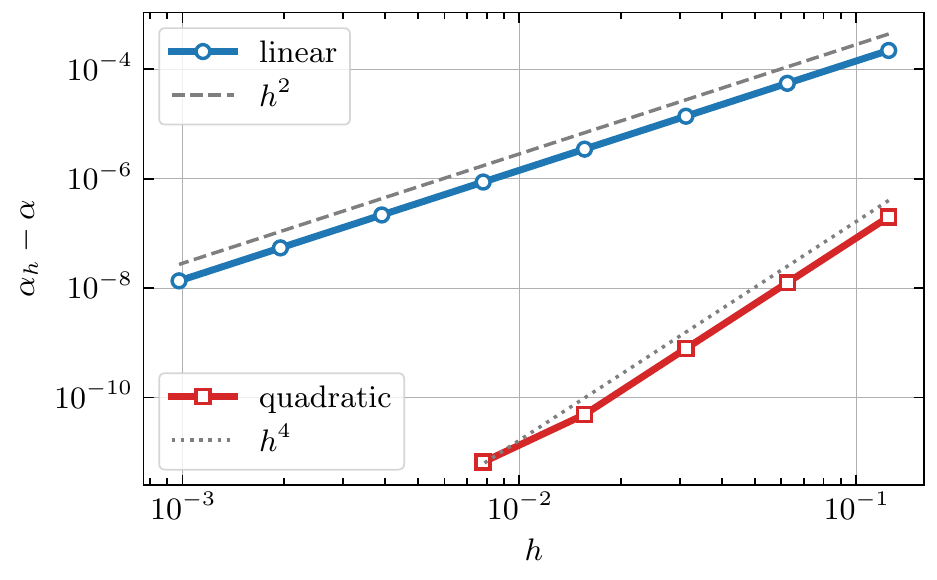}
	\caption{Convergence of discrete coercivity constant for ODE-IVP~\cref{eq:example_1} using linear and quadratic polynomials.}\label{fig:ode-eig-linear}
\end{figure}

\subsection{Advection-Diffusion-Reaction BVP}

Next, we consider the advection-diffusion-reaction PDE
in~\cref{eq:pde-advec-diff} with domain $\Omega = (0,1)^2$ and vector field
$\vb(\vx) = \left[ \frac{x}{2}, \frac{y}{2} \right]$.

Recall that the function $\phi := Ku - u = Cu \in H_0^1(\Omega)$ is characterized by~\cref{eq:example_3_phi_pde}.  With our choice of $\vb(\vx)$, this simplifies to
\begin{equation}
	-\Delta \phi + \frac{1}{2}\phi = \frac{1}{2}u.
\end{equation}

This PDE has a unique solution in $H^3(\Omega) \cap H_0^1(\Omega)$, and we can write
\begin{equation}
	Cu = \phi = \frac{1}{2}\left(-\Delta + \frac{1}{2}I \right)^{-1}u.
\end{equation}

Thus, if $\lambda$ is an eigenvalue of $C$, it follows that there is a $0\neq u \in H_0^1(\Omega)$ such that
\begin{subequations}
\begin{align}
&\frac{1}{2}\left(-\Delta + \frac{1}{2}I \right)^{-1}u = \lambda u,\\
&\frac{1}{2}u = \lambda \left(-\Delta + \frac{1}{2}I \right)u.
\end{align}
\end{subequations}
It follows that $\lambda \neq 0$, and
\begin{equation}\label{eq:advec-diff-eigen}
	-\Delta u = \left(\frac{1 - \lambda}{2\lambda} \right)u.
\end{equation}
For $m,n \in \mathbb{N}$, let $\beta_{mn} = (m^2 + n^2)\pi^2$ be the eigenvalues of the Laplace operator on the unit square.  From~\cref{eq:advec-diff-eigen}, we see that the eigenvalues of $C$ are $\lambda_{mn} = (1 + 2\beta_{mn})^{-1}$, and thus the eigenvalues of $K$ are of the form
\begin{equation}
	1 + \lambda_{mn} = \frac{2 + 2\beta_{mn}}{1 + 2\beta_{mn}},
\end{equation}
 and have the corresponding eigenspaces
\[
  \spn\left\{\sin(\pi x)\sin (\pi y) \right \}.
\]
For the unit square $\Omega$ and for this choice of $\vb(\vx)$, the coercivity constant is thus
\begin{equation}
	\alpha = \inf_{0 \neq u\in V} \frac{(\nabla u, \nabla u)_0 + (u,u)_0 + (\vb \cdot \nabla u, u)_0 }{(u,u)_0 + (\nabla u, \nabla u)_0}
	 = \frac{1 + 2\beta_{11}}{2 + 2\beta_{11}} = \frac{1 + 4\pi^2}{2 + 4\pi^2},
\end{equation}
with an identical eigenspace to the Laplacian.

In~\cref{fig:advec-diff-eig-linear}, we show the convergence of the coercivity constant using piecewise linear and quadratic polynomials. As with the previous example, we observe twice the rate of convergence as for function approximation, as expected.
\begin{figure}[!ht]
	\centering
	\includegraphics{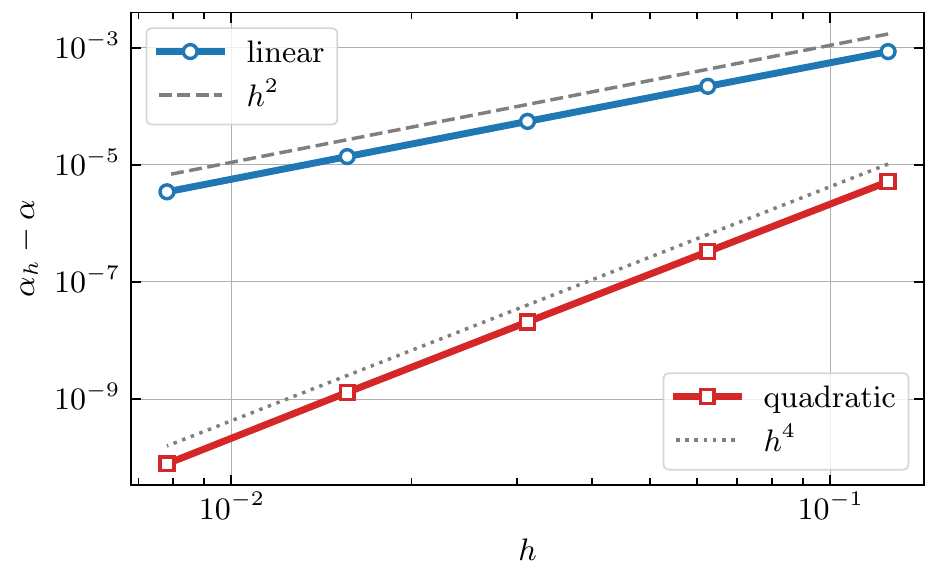}
	\caption{Convergence of discrete coercivity constant for the advection-diffusion BVP~\cref{eq:pde-advec-diff} using linear and quadratic polynomials.}\label{fig:advec-diff-eig-linear}
\end{figure}

\subsection{Least-squares Poisson}\label{sec:ls_poisson}

As a final numerical example, we consider the two least-squares formulations of
the Poisson equation~\cref{eq:ls-poisson,eq:ls-poisson-scaled} on the unit square.  See Appendix B for the derivation of the exact values of the coercivity constants.

Letting $\beta = 2\pi^2$ be the smallest eigenvalue of the Laplacian, the coercivity constant for the standard formulation~\cref{eq:ls-poisson} is
\begin{equation}\label{eq:coercivity_standard}
	\alpha = \frac{1 + 2\beta - \sqrt{1 + 4\beta}}{2(1 + \beta)} \approx 0.7603,
\end{equation}
and the coercivity constant of the rescaled formulation~\cref{eq:ls-poisson-scaled} is
\begin{equation}\label{eq:coercivity_rescaled}
\tilde{\alpha} = \frac{1 + 5\beta - \sqrt{(1 + \beta)(1 + 9\beta)}}{2(1 + \beta)}\approx 0.936.
\end{equation}
In either case, the corresponding eigenspace for the scalar variable $u$ is identical to the eigenspace of the Laplacian:
\begin{equation}
\spn\{\sin(\pi x)\sin(\pi y)\}.
\end{equation}
The ``eigenflux'' vector for the standard formulation
and for the rescaled formulation belong to the spaces
\begin{equation}
\spn\left\{\frac{-2}{1 + \sqrt{1 + 4\beta}}\nabla u\right\}
\quad\text{and}\quad
\spn\left\{\frac{2\tilde{\alpha}}{(1 + \beta)(1 - \tilde{\alpha})}\nabla u\right\}.
\end{equation}

For both problems, we approximate the scalar $u$ using piecewise linear finite
elements, and approximate the vector $\vq$ using lowest-order Raviart-Thomas
elements.
\Cref{fig:ls-poisson-eig} shows the convergence of the discrete coercivity
constant for both cases, highlighting converge as $\mathcal{O}(h^2)$.
We observe no qualitative difference between the different scalings
in terms of convergence behavior.  This indicates that an operator of the form
$K = D + C$ poses no additional numerical difficulties as compared to $K = I +
C$.
\begin{figure}[!ht]
	\centering
	\includegraphics{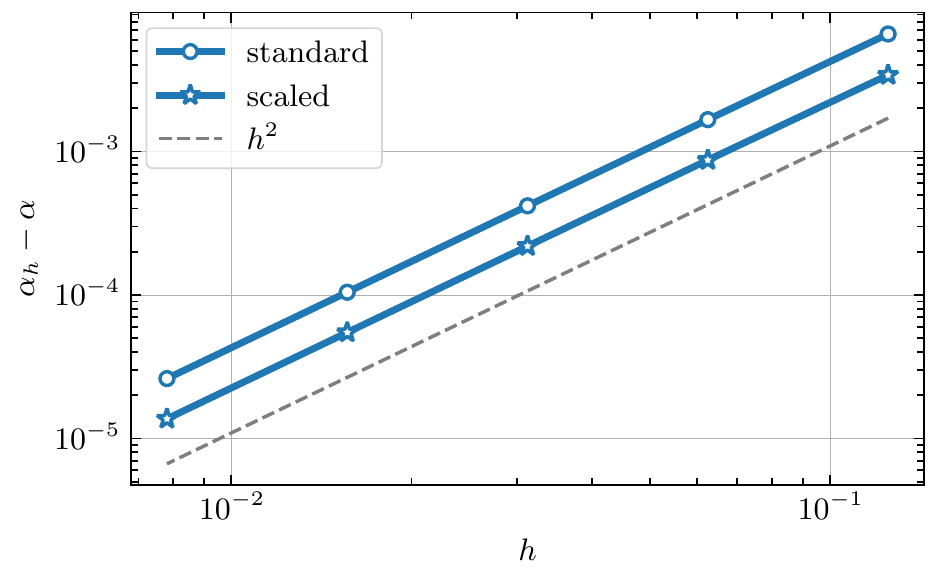}
  \caption{Convergence of discrete coercivity constant for the least-squares Poisson formulation corresponding to~\cref{eq:ls-poisson} and~\cref{eq:ls-poisson-scaled} using linear polynomials.}\label{fig:ls-poisson-eig}
\end{figure}

\section{Conclusion}\label{sec:conclusion}
The results shown in this paper quantify the convergence rates of discrete approximations to coercivity constants for a variety of variationally-posed differential equations.  The key observation leading to these results is that the coercivity constant is an eigenvalue of an operator $K = I + C$, where $C$ is compact.  Several numerical results illustrate the convergence theory.

Error bounds constructed for reduced basis methods rely on a lower-bound of the coercivity constant; for problems without a known lower bound for the constant of the continuous problem, convergence of discrete approximations are required in order to bound the error with respect to an analytical solution.  Thus, the results in this paper are particularly relevant to the method proposed in~\cite{chaudhry2020leastsquares}, which develops a reduced basis method with analytical error estimate.

There are numerous directions for future work.  Direct application of discrete coercivity constant convergence within a reduced basis method is a logical next step.  Investigation of the conditions for which an abstract variational problem leads to an operator $K = I + C$ is needed to establish the scope of the theory.  The rescaled differential equation in~\Cref{sec:poisson-rescaled} leads to an expression for $K$ with a diagonal operator, yet exhibits the same order of convergence as the originally scaled problem.  A heuristic explanation was provided; rigorously examining this case, and other problems leading to more general operators will expand the application of the theory to more general differential equations.  Finally, many variational problems of interest are not coercive, but inf-sup stable.  An extension to the approximation of the inf-sup constant will broaden the applicability of the convergence theory.

\section{Appendix A\@: The Choice of Adjoint}\label{sec:appendixA}

In~\cref{sec:func_analysis}, we discussed two definitions of adjoints for bounded linear operators;~\cref{eq:banach_adjoint} or~\cref{eq:banach_adjoint_pivot} are valid for bounded linear operators on normed spaces, and the Hilbert-adjoint is defined for bounded linear operators on Hilbert spaces,~\cref{eq:hilbert_adjoint}.  The adjoint defined in~\cref{eq:banach_adjoint} was used in~\cref{eq:definition_A} to define $A$, which allowed us to characterize the coercivity constant as a spectral value.

There is a third notion of adjoint for unbounded densely-defined linear operators on Hilbert spaces~\cites{arnold2018finite, conway2019course,kreyszig1978introductory}.  Let $\LL: \mathcal{D}(\LL) \subset W_1 \to W_2$ be a linear operator; here, $W_1$ and $W_2$ are Hilbert spaces, and $\mathcal{D}(\LL)$ is the domain of $\LL$.  If $\mathcal{D}(\LL)$ is dense in $W_1$, then there is a linear operator $\LL^\star:\mathcal{D}(\LL^\star) \subset W_2 \to W_1$ such that
\begin{equation}\label{eq:dense_adjoint}
	(\LL w_1, w_2)_{W_2} = (w_1, \LL^\star w_2)_{W_1}\quad \forall w_1\in \mathcal{D}(\LL),\ \forall w_2\in \mathcal{D}(\LL^\star ).
\end{equation}

For example, if $W_1 = W_2 = L^2(\Omega)$, with $\Omega$ an open bounded set, let $\LL$ be a linear differential operator with densely-defined domain satisfying $H_0^1(\Omega) \subset \mathcal{D}(\LL) \subset H^1(\Omega)$.  Then~\cref{eq:dense_adjoint} expresses an integration by parts identity with $H_0^1(\Omega) \subset \mathcal{D}(\LL^\star ) \subset H^1(\Omega)$.  In the context of linear differential equations, this adjoint is often known as the \emph{formal adjoint}~\cites{aubin2007approximation, cai2001first, evans2010partial}.

Because of its connection with differential equations, it may seem reasonable to alter the definition~\cref{eq:definition_A} using formal adjoints:
\begin{equation}\label{eq:definition_A_formal}
	\hat{A} = \frac{1}{2}\left(\M^\star \LL + \LL^\star \M \right).
\end{equation}
Unfortunately, this formulation does \textit{not} generally lead to a correct characterization of the coercivity constant.  We demonstrate this with the example in~\cref{eq:example_1}.

For this problem, $W_1 = W_2 = L^2(0,1)$, $\LL = \M = \frac{\dif }{\dif x} - I_X$ and the densely-defined domain is $\mathcal{D}(\LL) = V = \{ u\in H^1(0,1) : u(0) = 0\}$.

For $u,v \in V$, integration by parts shows that
\begin{equation}
	(\LL u, v)_0 = (u' - u, v)_0 = u(1)v(1) + (u,-v' - v)_0.
\end{equation}
Therefore, if $v\in \mathcal{D}(\LL^\star )$, it must hold that $v(1) = 0$.  In this case, we have $\LL^\star  = -\left(\frac{\dif }{\dif x} + I_X\right)$, with domain $\mathcal{D}(\LL^\star ) = V^\star  = \{ v\in H^1(0,1) : v(1) = 0\}$.

In order for $\hat{A}u = \LL^\star \LL u$ to be defined, we must have $u\in \mathcal{D}(\LL) = V$ and $\LL u \in \mathcal{D}(\LL^\star ) = V^\star $.  That is,
\begin{equation}\label{eq:domain_A_hat}
	\mathcal{D}(\hat{A}) = \left\{ u\in H^2(0,1) : u(0) = 0,\ u'(1) - u(1) = 0 \right\} \subsetneq V.
\end{equation}
The domain of $\hat{A}$ is now a proper subspace of $V$, in contrast to $A$ from~\cref{eq:definition_A}, which is defined on all of $V$.  We demonstrate that even if the operators $\hat{K} = \hat{A}^{-1}R_V$ and $\hat{K}^{-1}$ are well-defined on some subspace of $V$, the domain restriction prevents us from characterizing the coercivity constant through the spectrum of either operator.

For any $u \in \mathcal{D}(\hat{A})$, $\hat{A}u = \LL^\star \LL u = -\frac{\dif^2}{\dif x^2} + I_X$.  Thus, if $\hat{K}u = \lambda u \iff R_V u = \lambda\LL^\star \LL u$, it must hold that $u$ belongs to the subspace~\cref{eq:domain_A_hat} and
\begin{equation}
	\lambda (u - u'', v)_0 = (u,v)_0 + (u',v')_0,\quad \forall v\in V.
\end{equation}
An integration by parts shows that
\begin{subequations}\label{eq:ode_formal_adj}
	\begin{align}
		u - u'' &= \lambda(u - u''),\\
		u'(1) &= 0,\label{eq:bc_1}
	\end{align}
\end{subequations}
and since $u\in \mathcal{D}(\hat{A})$, two additional boundary conditions must be satisfied:
\begin{subequations}\label{eq:bc_formal_adj}
	\begin{align}
		u(0) &= 0,\\
		u'(1) - u(1) &= 0.\label{eq:bc_2}
	\end{align}
\end{subequations}
Note that~\cref{eq:bc_1} and~\cref{eq:bc_2} imply $u(1) = 0$.

If $\lambda = 1$, then~\cref{eq:ode_formal_adj} and~\cref{eq:bc_formal_adj} are satisfied for any $u\in H^2(0,1)$ with $0 = u(0) = u(1) = u'(1)$.  Thus, $\lambda = 1$ is an eigenvalue of $\hat{K}$.  Comparing with~\cref{eq:ex1_eig1}, observe that one eigenvalue has been correctly identified, although only a subspace of the eigenvectors are identified.

If $\lambda \neq 1$, then in order for~\cref{eq:ode_formal_adj} and~\cref{eq:bc_formal_adj} to hold, we must have
\begin{subequations}
	\begin{align}
		u - u'' &= 0,\\
		u(0) = u(1) = u'(1) &= 0,
	\end{align}
\end{subequations}
for which the only solution is $u(x) = 0$.  Thus, we cannot identify any non-unit eigenvalues of $\hat{K}$, and the eigenvalue and eigenspace of~\cref{eq:ex1_eig2} cannot be identified.  This is precisely the eigenvalue corresponding to the coercivity constant.

If $\hat{K}^{-1}u = \lambda u \iff \LL^\star\LL u = \lambda R_V$, we must have $u\in \mathcal{D}(\hat{A})$, and
\begin{equation}
	(u - u'', v)_0 = \lambda(u,v)_0 + \lambda(u',v')_0,\quad \forall v \in V.
\end{equation}
Integrating by parts and accounting for the boundary conditions in~\cref{eq:domain_A_hat}, we are led to the differential equation
\begin{subequations}\label{eq:ode_formal_adj2}
	\begin{align}
		u - u'' &= \lambda(u - u''),\\
		\lambda u'(1) &= 0,\\
		u(0) = u'(1) - u(1) &= 0.
	\end{align}
\end{subequations}
Just as for $\hat{K}$,~\cref{eq:ode_formal_adj2} holds for $\lambda = 1$ if $u\in H^2(0,1)$ and $0 = u(0) = u(1) = u'(1)$.  If $\lambda \neq 1$, then~\cref{eq:ode_formal_adj2} is only satisfied for $u(x) = 0$.  Once again, the eigenvalue corresponding to the coercivity constant cannot be identified.

\section{Appendix B\@: First-Order Formulation of Poisson Equation}

\subsection{Standard Scaling}\label{sec:standard}
We derive explicit representations for the eigenvalues and eigenfunctions of the operator in~\eqref{eq:K_standard}, which corresponds to the first-order system~\eqref{eq:ls-poisson}, which is equivalent to the Poisson equation with homogeneous Dirichlet conditions.

Since $K = I + C$, we can determine the eigenvalues of $K$ as $1 + \lambda$, where $\lambda$ is an eigenvalue of $C$.

If $\lambda$ is an eigenvalue of $C$, then there exists $(\mathbf{0}, 0) \neq (\vq, u) \in H(\text{div}) \times H_0^1(\Omega)$ such that for $\phi := \nabla\cdot \vq$,
\begin{subequations}
\begin{align}
-\nabla\left(-\Delta \right)^{-2}\phi -\nabla\left(-\Delta \right)^{-1}\left(I + \left(-\Delta \right)^{-1} \right)u  &= \lambda \vq, \label{eq:standard_Ca}\\
\left(-\Delta \right)^{-1}\left(I + \left(-\Delta \right)^{-1} \right)\phi + \left(-\Delta \right)^{-1}\left(2I + \left(-\Delta \right)^{-1} \right)u &= \lambda u. \label{eq:standard_Cb}
\end{align}
\end{subequations}

Applying the divergence operator to~\cref{eq:standard_Ca} and the negative Laplacian to~\cref{eq:standard_Cb}, we obtain
\begin{subequations}
\begin{align}
\left(-\Delta \right)^{-1}\phi +\left(I + \left(-\Delta \right)^{-1} \right)u  &= \lambda \phi,\label{eq:standard_Cc}\\
\left(I + \left(-\Delta \right)^{-1} \right)\phi +\left(2I + \left(-\Delta \right)^{-1} \right)u &= \lambda(-\Delta) u. \label{eq:standard_Cd}
\end{align}
\end{subequations}
Note that $(-\Delta) u \in L^2(\Omega)$ by~\cref{eq:range_C}.

Combining~\cref{eq:standard_Cc} and~\cref{eq:standard_Cd}, we obtain
\begin{align}\label{eq:simplified_C}
\begin{split}
	(\lambda + 1)\phi - \left(I + \left(-\Delta \right)^{-1} \right)u &= \lambda(-\Delta) u - \left(2I + \left(-\Delta \right)^{-1} \right)u\\
	\implies\quad (\lambda + 1)\phi &= \lambda(-\Delta) u - u.
	\end{split}
\end{align}

If $\lambda = -1$, then~\cref{eq:simplified_C} implies that $-\Delta u = -u$.  Since $-\Delta$ is positive-definite on $H_0^1(\Omega)$, it follows that $u = 0$.  Thus, from~\cref{eq:standard_Cc}, we have $\left(-\Delta \right)^{-1}\phi = -\phi$.  Since $\left(-\Delta \right)^{-1}$ is positive-definite, we also find that $\phi = \nabla\cdot\vq = 0$.  Finally, from~\cref{eq:standard_Ca}, it follows that $\vq = \mathbf{0}$.  This shows that $\lambda\neq -1$, which agrees with our assertions in the beginning of~\cref{sec:apps}.

If $\lambda = 0$, then~\cref{eq:simplified_C} becomes
\begin{equation}\label{eq:u_divq}
	-u = \phi = \nabla\cdot\vq.
\end{equation}

Plugging~\cref{eq:u_divq} into~\cref{eq:standard_Cc}, it follows that
\begin{equation}
	u = \phi =  \nabla\cdot\vq = 0,
\end{equation}
so that $\vq \in \mathcal{N}(\nabla \cdot)$, the space of divergence-free functions.  With $\vq \in \mathcal{N}(\nabla \cdot)$ and $u = 0$, it is clear that~\cref{eq:standard_Ca} and~\cref{eq:standard_Cb} hold for $\lambda = 0$.  Thus, $\lambda = 0$ is an eigenvalue of $C$ with infinite-dimensional eigenspace $\mathcal{N}(\nabla \cdot) \times \left\{ 0\right\}$.

Now, if $\lambda\neq 0, -1$, we can solve for $\phi$ using~\cref{eq:simplified_C} as
\begin{equation}\label{eq:standard_phi}
	\phi = \frac{-1}{\lambda + 1}u + \frac{\lambda}{\lambda + 1}(-\Delta) u.
\end{equation}

Substituting~\cref{eq:standard_phi} in~\cref{eq:standard_Cc} and simplifying, we obtain
\begin{equation}\label{eq:standard_no_phi}
	\frac{\lambda^2}{\lambda + 1}(-\Delta) u = \frac{3\lambda + 1}{\lambda + 1}u + \frac{\lambda}{\lambda + 1}(-\Delta)^{-1} u.
\end{equation}

Applying $(-\Delta)^{-1}$ to~\cref{eq:standard_no_phi}, we arrive at the equality
\begin{equation}\label{eq:standard_u_eigen}
	\frac{\lambda^2}{\lambda + 1} u = \frac{3\lambda + 1}{\lambda + 1}(-\Delta)^{-1}u + \frac{\lambda}{\lambda + 1}(-\Delta)^{-2} u.
\end{equation}

Since $(-\Delta)^{-1}$ is compact on $H_0^1(\Omega)$, it follows from the Spectral Mapping Theorem~\cites{kreyszig1978introductory, conway2019course} that $u$ is an eigenfunction of $(-\Delta)^{-1}$.  Thus, if $\beta_{mn} = (m^2 + n^2)\pi^2$ ($m,n\in \mathbb{N}$), are the eigenvalues of the Laplacian on $\Omega = (0,1)^2$ with eigenfunctions $u_{mn} \in \text{span}\left\{\sin(\pi x)\sin(\pi y) \right\}$, we have
\begin{equation}\label{eq:laplace_eigen}
	(-\Delta)^{-1}u = (-\Delta)^{-1}u_{mn} = \frac{1}{\beta_{mn}}u_{mn}.
\end{equation}

Using this result in~\cref{eq:standard_u_eigen}, we obtain the relation
\begin{equation}
	 \frac{3\lambda + 1}{\beta_{mn}(\lambda + 1)} +  \frac{\lambda}{\beta_{mn}^2(\lambda + 1)} = \frac{\lambda^2}{\lambda + 1},
\end{equation}

which simplifies to the quadratic equation
\begin{equation}\label{eq:standard_quadratic}
	\beta_{mn}^2\lambda^2 - (1 + 3\beta_{mn})\lambda - \beta_{mn} = 0.
\end{equation}

It follows that the non-zero eigenvalues of $C$ are given by
\begin{equation}
	\lambda_{mn}^{\pm} = \frac{1 + 3\beta_{mn} \pm (1 + \beta_{mn})\sqrt{1 + 4\beta_{mn}}}{2\beta_{mn}^2}.
\end{equation}

Returning to~\cref{eq:standard_Ca}, by substituting the equalities~\cref{eq:laplace_eigen} and~\cref{eq:standard_phi}, we find that
\begin{equation}
	\vq = \vq^{\pm}_{mn} = - \frac{2\lambda_{mn}^{\pm}\beta_{mn} + \beta_{mn} + \lambda_{mn}^{\pm}}{\lambda_{mn}^{\pm}\beta_{mn}^2(\lambda_{mn}^{\pm} + 1)}\nabla u_{mn}.
\end{equation}
Simplifying with the help of~\cref{eq:standard_quadratic},
\begin{equation}
	\vq^{\pm}_{mn} = -\frac{2}{1 \pm \sqrt{1 + 4\beta_{mn}}}\nabla u_{mn}.
\end{equation}

We have characterized the spectrum of $C$, from which it follows that the eigenvalues of $K$ are
\begin{equation}\label{eq:standard_eigvals}
\mu_0 = 1,\quad \mu_{mn}^\pm = 1 + \lambda_{mn}^\pm,
\end{equation}
with corresponding eigenspaces
\begin{equation}
	\mathcal{N}(\nabla\cdot) \times \{ 0\},\quad \text{span}\left\{(\vq_{mn}^\pm, u_{mn})\right\}
\end{equation}

The eigenvalues of $K^{-1}$ are the reciprocals of~\eqref{eq:standard_eigvals}.    Some more algebra shows that
\begin{equation}\label{eq:standard_K_inv}
	\frac{1}{\mu_{mn}^\pm} = \frac{1}{1 + \lambda_{mn}^\pm} = \frac{1 + 2\beta_{mn} \mp \sqrt{1 + 4\beta_{mn}}}{2(1 + \beta_{mn})}.
\end{equation}
From~\cref{eq:standard_K_inv}, it is evident that the minimum eigenvalue of $K^{-1}$ is given by~\cref{eq:coercivity_standard}.

\subsection{Rescaled Equations}

We derive explicit representations for the eigenvalues and eigenfunctions of the operator in~\eqref{eq:D_C}, which corresponds to the first-order system~\eqref{eq:ls-poisson-scaled}.  The derivation is similar to~\cref{sec:standard}

Since $K$ is not of the form $I + C$, we will compute the eigenvalues of $K$ directly.

If $\lambda$ is an eigenvalue of $K$, then there exists $(\mathbf{0}, 0) \neq (\vq, u) \in H(\text{div}) \times H_0^1(\Omega)$ such that for $\phi := \nabla\cdot \vq$,
\begin{subequations}
\begin{align}
\vq -\nabla\left(-\Delta \right)^{-2}\phi -\frac{1}{2}\nabla\left(-\Delta \right)^{-1}\left(I + \left(-\Delta \right)^{-1} \right)u  &= \lambda \vq, \label{eq:rescaled_Ka}\\
\frac{1}{4}u + \frac{1}{2}\left(-\Delta \right)^{-1}\left(I + \left(-\Delta \right)^{-1} \right)\phi + \frac{1}{2}\left(-\Delta \right)^{-1}\left(I + \frac{1}{2}\left(-\Delta \right)^{-1} \right)u &= \lambda u. \label{eq:rescaled_Kb}
\end{align}
\end{subequations}

Applying the divergence operator to~\cref{eq:rescaled_Ka} and the negative Laplacian to~\cref{eq:rescaled_Kb}, we obtain
\begin{subequations}
\begin{align}
\left(-\Delta \right)^{-1}\phi + \frac{1}{2}\left(I + \left(-\Delta \right)^{-1} \right)u  &= (\lambda - 1) \phi,\label{eq:rescaled_Kc}\\
\frac{1}{2}\left(I + \left(-\Delta \right)^{-1} \right)\phi +\frac{1}{2}\left(I + \frac{1}{2}\left(-\Delta \right)^{-1} \right)u &= \left(\lambda - \frac{1}{4}\right)(-\Delta) u. \label{eq:rescaled_Kd}
\end{align}
\end{subequations}

Combining~\cref{eq:rescaled_Kc} and~\cref{eq:rescaled_Kd}, we obtain
\begin{align}\label{eq:simplified_K}
\begin{split}
	\lambda\phi - \frac{1}{2}\left(I + \left(-\Delta \right)^{-1} \right)u &= 2\left(\lambda - \frac{1}{4}\right)(-\Delta) u - \left(I + \frac{1}{2}\left(-\Delta \right)^{-1} \right)u\\
	\implies\quad \lambda\phi &= 2\left(\lambda - \frac{1}{4}\right)(-\Delta) u - \frac{1}{2}u.
	\end{split}
\end{align}

If $\lambda = 0$, then~\cref{eq:simplified_K} implies that $-\Delta u = -u$, and the positive-definiteness of the Laplacian it follows that $u = 0$.  This result combined with~\cref{eq:rescaled_Kc} shows that $\left(-\Delta \right)^{-1}\phi = -\phi$; i.e. $\phi = \nabla\cdot\vq = 0$.  Finally,~\cref{eq:rescaled_Ka} shows that $\vq = \mathbf{0}$.  Thus, $\lambda\neq 0$.

If $\lambda = \frac{1}{4}$, then~\cref{eq:simplified_K} becomes $\phi = -2u$.  Substituting this expression into~\cref{eq:rescaled_Kc}, we obtain $\left(-\Delta \right)^{-1}u = -\frac{2}{3}u$.  By again appealing to the fact that the Laplacian is positive-definite, we obtain $u = \phi = \nabla\cdot \vq = 0$.  By~\cref{eq:rescaled_Ka}, it follows that once again, $\vq = \mathbf{0}$.  So $\lambda \neq \frac{1}{4}$.

If $\lambda = 1$, then~\cref{eq:simplified_K} becomes $\phi = -\frac{1}{2}u + \frac{3}{2}(-\Delta) u$.  Similarly, to the cases of $\lambda = 0$ and $\lambda = \frac{1}{4}$, combining this with~\cref{eq:rescaled_Kc} shows that $u = \phi = \nabla\cdot\vq = 0$.  However,~\cref{eq:rescaled_Ka} now reduces to $\vq	 = \vq$.  Thus, $\lambda = 1$ is an eigenvalue, with infinite-dimensional eigenspace $\mathcal{N}(\nabla\cdot) \times \{ 0\}$.

Proceeding with the assumption that $\lambda \neq 0, \frac{1}{4}, 1$, solving for $\phi$ in~\cref{eq:simplified_K} shows that
\begin{equation}\label{eq:rescaled_phi}
	\phi = -\frac{1}{2\lambda}u + \left(2 - \frac{1}{2\lambda}\right)(-\Delta)^{-1}u.
\end{equation}

Substituting~\cref{eq:rescaled_phi} into~\cref{eq:rescaled_Kc} and simplifying, we obtain
\begin{equation}\label{eq:rescaled_no_phi}
	\frac{(\lambda - 1)(4\lambda - 1)}{2\lambda}(-\Delta) u = \frac{3\lambda - 1}{\lambda}u + \frac{\lambda - 1}{2\lambda}(-\Delta)^{-1} u.
\end{equation}

Applying $(-\Delta)^{-1}$ to~\cref{eq:rescaled_no_phi}, we arrive at
\begin{equation}\label{eq:rescaled_u_eigen}
	\frac{(\lambda - 1)(4\lambda - 1)}{2\lambda} u = \frac{3\lambda - 1}{\lambda}(-\Delta)^{-1}u + \frac{\lambda - 1}{2\lambda}(-\Delta)^{-2} u.
\end{equation}

Just as in~\cref{sec:standard}, the compactness of $(-\Delta)^{-1}$ and the Spectral Mapping Theorem show that $u$ is an eigenfunction of $(-\Delta)^{-1}$; i.e.~\cref{eq:laplace_eigen} holds for $\beta_{mn} = (m^2 + n^2)\pi^2$ and $u = u_{mn} \in \text{span}\left\{\sin(\pi x)\sin(\pi y)\right\}$.

Using this in~\cref{eq:rescaled_u_eigen} and simplifying, we obtain the quadratic equation
\begin{equation}\label{eq:rescaled_quadratic}
	4\beta_{mn}^2\lambda^2 - (1 + \beta_{mn})(1 + 5\beta_{mn})\lambda + (1 +\beta_{mn})^2 = 0.
\end{equation}

It follows that the non-unit eigenvalues of $K$ are given by
\begin{equation}\label{eq:rescaled_eigvals}
	\lambda_{mn}^{\pm} = \frac{(1 + \beta_{mn})(1 + 5\beta_{mn}) \pm (1 + \beta_{mn})\sqrt{(1 + \beta_{mn})(1 + 9\beta_{mn})}}{8\beta_{mn}^2}.
\end{equation}

Returning to~\cref{eq:rescaled_Ka}, by substituting the equalities~\cref{eq:laplace_eigen} and~\cref{eq:rescaled_phi}, we find that
\begin{equation}
	\vq = \vq^{\pm}_{mn} = - \frac{5\lambda_{mn}^{\pm}\beta_{mn} - \beta_{mn} + \lambda_{mn}^{\pm} - 1}{2\lambda_{mn}^{\pm}\beta_{mn}^2(\lambda_{mn}^{\pm} - 1)}\nabla u_{mn}.
\end{equation}

Simplifying with the help of~\cref{eq:rescaled_quadratic},
\begin{equation}
	\vq^{\pm}_{mn} = \frac{2\lambda_{mn}^\pm}{(1 + \beta_{mn})(1 - \lambda_{mn}^\pm)}\nabla u_{mn}.
\end{equation}

The eigenvalues of $K^{-1}$ are the reciprocals of~\eqref{eq:rescaled_eigvals}.    Some more algebra shows that
\begin{equation}\label{eq:rescaled_K_inv}
	\frac{1}{\mu_{mn}^\pm} =
  \frac{1}{\lambda_{mn}^\pm} =
  \frac{1 + 5\beta_{mn} \mp \sqrt{(1 +\beta_{mn})(1 + 9\beta_{mn})}}{2(1 + \beta_{mn})}.
\end{equation}

From~\cref{eq:rescaled_K_inv}, it is evident that the minimum eigenvalue of $K^{-1}$ is given by~\cref{eq:coercivity_rescaled}.

\bibliographystyle{amsplain}
\bibliography{refs-coercivity}

\end{document}